\documentclass[11pt,a4paper]{amsart}
\usepackage{amssymb}
\usepackage[dvips,bookmarks,colorlinks=true,linkcolor=blue,urlcolor=red]{hyperref}
\numberwithin{equation}{section}
\newcommand{\D}{\displaystyle}

\newcommand{\car}{\mathbf{1}}
\newcommand{\R}{{\mathbb R}}

\newcommand{\Z}{{\mathbb Z}}
\newcommand{\N}{{\mathbb N}}

\newcommand{\esp}{{\mathbb E}}
\newcommand{\pro}{{\mathbb P}}

\newcommand{\dist}{{\rm dist}\,}

\newcommand{\vers}{\operatornamewithlimits{\to}}
\theoremstyle{plain}
\newtheorem{Th}{Theorem}[section]
\newtheorem{Le}{Lemma}[section]
\newtheorem{Pro}{Proposition}[section]

\theoremstyle{definition}
\newtheorem{Rem}{Remark}[section]
\title[Asymptotic ergodicity in the localized phase]{Asymptotic
  ergodicity of the eigenvalues of random operators in the localized
  phase}
\author{Fr{\'e}d{\'e}ric Klopp} 
\address[Fr{\'e}d{\'e}ric
Klopp]{LAGA, U.M.R. 7539 C.N.R.S, Institut Galil{\'e}e, Universit{\'e}
  Paris-Nord, 99 Avenue J.-B. Cl{\'e}ment, F-93430 Villetaneuse, France}
\email{\href{mailto:klopp@math.univ-paris13.fr}{klopp@math.univ-paris13.fr}}
\keywords{}
\subjclass{}
\begin{document}
%
\begin{abstract}
  We prove that, for a general class of random operators, the family
  of the unfolded eigenvalues in the localization region is
  asymptotically ergodic in the sense of N. Minami
  (see~\cite{Mi:11}). N. Minami conjectured this to be the case for
  discrete Anderson model in the localized regime. We also provide a
  local analogue of this result. From the asymptotics ergodicity, one
  can recover the statistics of the level spacings as well as a number
  of other spectral statistics. Our proofs rely on the analysis
  developed in~\cite{Ge-Kl:10}.
  \vskip.5cm\noindent \textsc{R{\'e}sum{\'e}.}  On d{\'e}montre que, pour une
  classe g{\'e}n{\'e}rale d'op{\'e}rateurs al{\'e}a\-toires, les familles valeurs
  propres ``d{\'e}pli{\'e}es'' sont asymptotiquement ergodiques au sens de
  N. Minami (voir~\cite{Mi:11}). N. Minami {\`a} conjectur{\'e} que ceci est
  vrai pour le mod{\`e}le d'Anderson discret dans le r{\'e}gime localis{\'e}.  On
  d{\'e}montre {\'e}galement un r{\'e}sultat analogue pour les valeurs propres
  ``locales''. L'ergodicit{\'e} asymptotique des valeurs propres permet
  alors d'en d{\'e}duire les statistiques des espacements de niveaux ainsi
  que nombre d'autres statistiques spectrales. Nos preuves reposent
  sur l'analyse faite dans~\cite{Ge-Kl:10}.
\end{abstract}
\thanks{The author is partially supported by the grant
  ANR-08-BLAN-0261-01.}
\setcounter{section}{-1}
\maketitle
\section{Introduction}
\label{intro}
On $\ell^2(\Z^d)$, consider the random Anderson model
\begin{equation*}
  H_{\omega}=-\Delta+\lambda V_\omega
\end{equation*}
where
\begin{itemize}
\item $-\Delta$ is the free discrete Laplace operator
  \begin{equation}
    \label{eq:29}
    (-\Delta u)_n=\sum_{|m-n|=1}u_m\quad\text{ for }
    u=(u_n)_{n\in\Z^d}\in\ell^2(\Z^d);
  \end{equation}
\item $V_\omega$ is the random potential
  \begin{equation}
    \label{eq:54}
    (V_\omega u)_n=\omega_n u_n\quad\text{ for }u=(u_n)_{n\in\Z^d}
    \in\ell^2(\Z^d).
  \end{equation}
  We assume that the random variables $(\omega_n)_{n\in\Z^d}$ are
  independent identically distributed and that their common
  distribution admits a compactly supported bounded density, say $g$.
\item The coupling constant $\lambda$ is chosen positive.
\end{itemize}
It is then well known (see e.g.~\cite{MR2509110}) that
\begin{itemize}
\item let $\Sigma:=[-2d,2d]+$supp$\,g$ and $S_-$ and $S_+$ be the
  infimum and supremum of $\Sigma$; for almost every
  $\omega=(\omega_n)_{n\in\Z^d}$, the spectrum of $H_{\omega}$ is
  equal to $\Sigma$;
\item there exists a bounded density of states, say $E\mapsto\nu(E)$,
  such that, for any continuous function $\varphi:\ \R\to\R$, one has
  \begin{equation}
    \label{eq:10}
    \int_\R\varphi(E)\nu(E)dE=
    \mathbb{E}(\langle\delta_0,\varphi(H_\omega)\delta_0\rangle).
  \end{equation}
  Here, and in the sequel, $\mathbb{E}(\cdot)$ denotes the expectation
  with respect to the random parameters, and $\pro(\cdot)$ the
  probability measure they induce.\\
  Let $N$ be the integrated density of states of $H_\omega$ i.e. $N$
  is the distribution function of the measure $\nu(E)dE$. The function
  $\nu$ is only defined $E$-almost everywhere. In the sequel, when we
  speak of $\nu(E)$ for some $E$, we mean that the non decreasing
  function $N$ is differentiable at $E$ and that $\nu(E)$ is its
  derivative at $E$.
\end{itemize}
For $L\in\N$, let $\Lambda=\Lambda_L=[-L,L]^d$ be a large box and
$|\Lambda|:=\#\Lambda=(2L+1)^d$ be its cardinality. Let
$H_{\omega}(\Lambda)$ be the operator $H_\omega$ restricted to
$\Lambda$ with periodic boundary conditions. The notation
$|\Lambda|\to+\infty$ is a shorthand for considering
$\Lambda=\Lambda_L$ in the limit $L\to+\infty$. Let us denote the
eigenvalues of $H_{\omega}(\Lambda)$ ordered increasingly and repeated
according to multiplicity by $E_1(\omega,\Lambda)\leq
E_2(\omega,\Lambda)\leq \cdots\leq E_{|\Lambda|}(\omega,\Lambda)$.
\par For $t\in[0,1]$, consider the following point process
\begin{equation}
  \label{eq:13}
  \Xi(\omega,t,\Lambda) = 
  \sum_{n=1}^{|\Lambda|} \delta_{|\Lambda|[N(E_n(\omega,\Lambda))-t]}.
\end{equation}
We prove
\begin{Th}
  \label{thr:1}
  For sufficiently large coupling constant $\lambda$, $\omega$-almost
  surely, when $|\Lambda|\to+\infty$, the probability law of the point
  process $\Xi(\omega,\cdot,\Lambda)$ under the uniform distribution
  $\car_{[0,1]}(t)dt$ converges to the law of the Poisson point
  process on the real line with intensity $1$.
\end{Th}
\noindent This proves in particular a conjecture by N.~Minami
(see~\cite{Mi:06,Mi:11}); a weaker version of Theorem~\ref{thr:1},
namely,
$L^2$-conver\-gence in $\omega$ when $d=1$, is proved in~\cite{Mi:11}.\\
Theorem~\ref{thr:1}, in particular, implies the convergence of the
level spacings statistics already obtained for this model under more
restrictive assumptions in~\cite{Ge-Kl:10} (see also
Theorem~\ref{thr:7} in the present paper for more details). Indeed, in
Theorem~\ref{thr:1}, we do not make any regularity assumption on the
distribution of the random variables except for their having a common
bounded compactly supported
density.\\
Actually, Theorem~\ref{thr:1} is a prototype of the general result we
state and prove below. Essentially, we prove that the claim in
Theorem~\ref{thr:1} holds in the localization region for any random
Hamiltonian satisfying a Wegner and a Minami estimate (see assumptions
(W) and (M) in section~\ref{sec:results}). To do so, we use the
analysis made in~\cite{Ge-Kl:10}; in particular, our analysis relies
on one of the approximation theorems proved in~\cite{Ge-Kl:10},
namely, Theorem~1.16.
\section{The results}
\label{sec:results}
Consider $H_\omega=H_0+V_\omega$, a $\Z^d$-ergodic random
Schr{\"o}\-dinger operator on $\mathcal{H}=L^2(\R^d)$ or $\ell^2(\Z^d)$
(see e.g.~\cite{MR94h:47068,MR1935594}). Typically, the background
potential $H_0$ is the Laplacian $-\Delta$, possibly perturbed by a
periodic potential. Magnetic fields can be considered as well; in
particular, the Landau Hamiltonian is also admissible as a background
Hamiltonian. For the sake of simplicity, we assume that $V_\omega$ is
almost surely bounded; hence, almost surely, $H_\omega$ have the same
domain $H^2(\R^d)$ or $\ell^2(\Z^d)$.
\subsection{The setting and the assumptions}
\label{sec:setting-assumptions}
For $\Lambda$, a cube in either $\R^d$ or $\Z^d$, we let
$H_\omega(\Lambda)$ be the self-adjoint operator $H_\omega$ restricted
to $\Lambda$ with periodic boundary conditions. As in~\cite{Ge-Kl:10},
our analysis stays valid for Dirichlet boundary conditions.\\
Furthermore, we shall denote by $\car_J(H)$ the spectral projector of
the operator $H$ on the energy interval $J$. $\esp(\cdot)$ denotes the
expectation with respect to $\omega$.\\
Our first assumption will be an independence assumption for local
Hamiltonians that are far away from each other, that is,
\begin{description}
\item[(IAD)] There exists $R_0>0$ such that for any two cubes
  $\Lambda$ and $\Lambda'$ such that dist$(\Lambda,\Lambda')>R_0$, the
  random Hamiltonians $H_\omega(\Lambda)$ and $H_\omega(\Lambda')$ are
  stochastically independent.
\end{description}
\begin{Rem}
  This assumption may be relaxed to assume that the correlation
  between the random Hamiltonians $H_\omega(\Lambda)$ and
  $H_\omega(\Lambda')$ decays sufficiently fast as
  dist$(\Lambda,\Lambda')\to+\infty$. We refer to~\cite{Ge-Kl:10} for
  more details.
\end{Rem}
Let $\Sigma$ be the almost sure spectrum of $H_\omega$. Pick $I$ a
relatively compact open subset of $\Sigma$. Assume the following
holds:
\begin{description}
\item[(W)] a Wegner estimate holds in $I$, i.e. there exists $C>0$
  such that, for $J\subset I$, and $\Lambda$, a cube in $\R^d$ or
  $\Z^d$, one has
  \begin{equation}
    \label{eq:1}
    \esp\left[\text{tr}(\car_J(H_\omega(\Lambda)))
    \right]\leq C |J|\,|\Lambda| .
  \end{equation}
\item[(M)] a Minami estimate holds in $I$, i.e. there exists $C>0$ and
  $\rho>0$ such that, for $J\subset I$, and $\Lambda$, a cube in
  $\R^d$ or $\Z^d$, one has
  \begin{equation}
    \label{eq:2}
    \esp\left[\text{tr}(\car_J(H_\omega(\Lambda)))
      \cdot[\text{tr}(\car_J(H_\omega(\Lambda)))-1]\right]\leq
    C (|J|\,|\Lambda|)^{1+\rho}.
  \end{equation}
\end{description}
\begin{Rem}
  \label{rem:1}
  The Wegner estimate (W) has been proved for many random
  Schr{\"o}\-dinger models e.g. for both discrete and continuous Anderson
  models under rather general conditions on the single site potential
  and on the randomness (see
  e.g.~\cite{MR2509108,MR2509110,MR2307751,MR2378428}) but also for
  other models (see e.g.~\cite{MR2658987,1007.2483}). The right hand
  side in~\eqref{eq:1} can be lower bounded by the probability to have
  at least one eigenvalue in $J$ (for $J$
  small).\\
  Weaker forms of assumption (W) i.e. when the right hand side is
  replaced with $C |J|^\alpha\,|\Lambda|^\beta$ for some
  $\alpha\in(0,1]$ and $\beta\geq1$, are known to hold also for some
  non monotonous models (see
  e.g.~\cite{MR95m:82080,MR1934351,MR2423576}). This is sufficient for
  our proofs to work if one additionnally knows that the integrated
  density of states is absolutely continuous.  \vskip.2cm
  \noindent On the Minami estimate (M), much less is known: in any
  dimension, it holds for the discrete Anderson model with $I=\Sigma$
  (see~\cite{MR97d:82046,MR2290333,MR2360226,MR2505733}). For the
  continuous Anderson model in any dimension, in~\cite{MR2663411}, it
  is shown to hold at the bottom of the spectrum under more
  restrictive conditions on the single site potential than needed to
  prove the Wegner estimate (W). These proofs yield an optimal
  exponent $\rho=1$. The right hand side in~\eqref{eq:2} can be lower
  bounded by the probability to have at least two eigenvalues in $J$.
  So, (M) can be interpreted as a measure of the independence of close
  by eigenvalues.
\end{Rem}
The integrated density of states is defined as
\begin{equation}
  \label{eq:83}
  N(E):=\lim_{|\Lambda|\to+\infty}\frac{\#\{\text{e.v. of
    }H_\omega(\Lambda)\text{ less than E}\}}{|\Lambda|}
\end{equation}
By (W), $N(E)$ is the distribution function of a measure that is
absolutely continuous with respect to to the Lebesgue measure on
$\R$. Let $\nu$ be the density of state of $H_\omega$ i.e. the
distributional derivative of $N$. In the sequel, for a set $I$,
$|N(I)|$ denotes the Lebesgue measure of $N(I)$ i.e.
$\D|N(I)|=\int_I\nu(E)dE$.\vskip.1cm\noindent
Let us now describe what we call the localized regime in the
introduction. For $L\geq1$, $\Lambda_L$ denotes the cube
$[-L/2,L/2]^d$ in either $\R^d$ or $\Z^d$. In the sequel, we write
$\Lambda$ for $\Lambda_L$ i.e. $\Lambda=\Lambda_L$ and when we write
$|\Lambda|\to+\infty$, we mean $L\to+\infty$.\\
Let $\mathcal{H}_\Lambda$ be $\ell^2(\Lambda\cap\Z^d)$ in the discrete
case and $L^2(\Lambda)$ in the continuous one. For a vector
$\varphi\in\mathcal{H}$, we define
\begin{equation}
  \label{eq:22}
  \|\varphi\|_x=
  \begin{cases}
    \|\car_{\Lambda(x)}\varphi\|_2\text{ where }\Lambda(x)=\{y;\
    |y-x|\leq1/2\}&\quad\text{
      if }\mathcal{H}=L^2(\R^d),\\
    \quad\quad|\varphi(x)|&\quad\text{ if }\mathcal{H}=\ell^2(\Z^d).
  \end{cases}
\end{equation}
Let $I$ be a compact interval. We assume that $I$ lies in the region
of complete localization (see e.g.~\cite{MR2078370,MR2203782}) for
which we use the following finite volume version:
\begin{description}
\item[(Loc)] for all $\xi\in(0,1)$, one has
  \begin{equation}
    \label{eq:84}
    \sup_{L>0}\,\sup_{\substack{\text{supp} f\subset I \\ |f|\leq1}}\,
    \esp\left(\sum_{\gamma\in\Z^d} e^{|\gamma|^\xi}\,
      \|\car_{\Lambda(0)}f(H_\omega(\Lambda_L))
      \car_{\Lambda(\gamma)}\|_2\right)<+\infty.
  \end{equation}
\end{description}
\begin{Rem}
  \label{rem:2}
  Such a region of localization has been shown to exist and described
  for many random models (see
  e.g.~\cite{MR2203782,MR2002h:82051,MR2207021,MR1935594,MR95m:82080,MR1934351,MR2423576,MR2658987,1007.2483});
  a fairly recent review can be found in~\cite{MR2509110}; other
  informational texts include~\cite{MR94h:47068,MR2078370}.\\
  Once a Wegner estimate is known (though it is not an absolute
  requirement see e.g.~\cite{MR2180453,MR2314108,MR2352267}), the
  typical regions where localization holds are vicinities of the edges
  of the spectrum. One may have localization over larger regions (or
  the whole) of the spectrum if the disorder is large like in
  Theorem~\ref{thr:1}.\\
  This assumption (Loc) may be relaxed; we refer to Remark 1.3
  of~\cite{Ge-Kl:10} for more details.
\end{Rem}
For $L\in\N$, recall that $\Lambda=\Lambda_L$ and that
$H_{\omega}(\Lambda)$ is the operator $H_\omega$ restricted to
$\Lambda$ with periodic boundary conditions. The notation
$|\Lambda|\to+\infty$ is a shorthand for considering
$\Lambda=\Lambda_L$ in the limit $L\to+\infty$.\\
Finally, let $E_1(\omega,\Lambda)\leq E_2(\omega,\Lambda)\leq
\cdots\leq E_N(\omega,\Lambda)\leq\cdots $ denote the eigenvalues of
$H_\omega(\Lambda)$ ordered increasingly and repeated according to
multiplicity.\vskip.2cm
\noindent We state our results in two cases. In the first case
described in section~\ref{sec:macr-energy-interv}, we consider a
macroscopic energy interval i.e. the energy interval in which we study
the eigenvalues is a fixed compact interval where all the above
assumptions hold. In the second case described in
section~\ref{sec:micr-energy-interv}, the energy interval shrinks to a
point but not too fast so as to contain enough
eigenvalues that is asymptotically infinitely many eigenvalues.\\
We also consider another point of view on the random
Hamiltonian. Namely, under assumption (Loc), in $I$, one typically
proves that the spectrum is made only of eigenvalues and that to these
eigenvalues, one associates exponentially decaying eigenfunctions
(exponential or Anderson localization) (see
e.g.~\cite{MR94h:47068,MR2078370,MR2203782,MR2509110}). One can then
enumerate these eigenvalues in an energy interval by considering only
those with localization center (i.e. with most of their mass) in some
cube $\Lambda$ and study the thus obtained process. This is done in
section~\ref{sec:results-full-random}.
\subsection{Macroscopic energy intervals}
\label{sec:macr-energy-interv}
For $J=[a,b]$ a compact interval such that $N(b)-N(a)=|N(J)|>0$ and a
fixed configuration $\omega$, consider the point process
\begin{equation}
  \label{eq:6}
  \Xi_J(\omega,t,\Lambda) = 
  \sum_{E_n(\omega,\Lambda)\in J}
  \delta_{|N(J)||\Lambda|[N_J(E_n(\omega,\Lambda))-t]}
\end{equation}
under the uniform distribution in $[0,1]$ in $t$; here we have set
\begin{equation}
  \label{eq:7}
  N_J(\cdot):=\frac{N(\cdot)-N(a)}{N(b)-N(a)}=\frac{N(\cdot)-N(a)}{|N(J)|}.  
\end{equation}
Our main result is
\begin{Th}
  \label{thr:2}
  Assume (IAD), (W), (M) and (Loc) hold. Assume that $J\subset I$, the
  localization region, is such that $|N(J)|>0$.\\
  Then, $\omega$-almost surely, the probability law of the point
  process $\Xi_J(\omega,\cdot,\Lambda)$ under the uniform distribution
  $\car_{[0,1]}(t)dt$ converges to the law of the Poisson point
  process on the real line with intensity $1$.
\end{Th}
\noindent First, let us note that Theorem~\ref{thr:1} is an immediate
consequence of Theorem~\ref{thr:2} as it is well known that, for the
discrete Anderson model at large disorder, the whole spectrum is
localized in the sense of (Loc) (see e.g.~\cite{MR2509110}).\\
A number of spectral statistics for the sequence of unfolded
eigenvalues are immediate consequences of Theorem~\ref{thr:2} and the
results of~\cite{MR2352280}. For example, by Proposition 4.4
of~\cite{MR2352280}, it implies the convergence of the empirical
distribution of unfolded level spacings to $e^{-x}$
(see~\cite{MR2352280,Mi:11,Ge-Kl:10}). We refer to~\cite{MR2352280}
for more results on the statistics of asymptotically ergodic sequences.\\
As in~\cite{Ge-Kl:10}, one can also study the statistics of the levels
themselves i.e. before unfolding. Using classical results on
transformations of point processes (see~\cite{MR1700749,MR2364939})
and the fact that $N$ is Lipschitz continuous and increasing, one
obtains
\begin{Th}
  \label{thr:8}
  Assume (IAD), (W), (M) and (Loc) hold. Assume that $J=[a,b]\subset
  I$ is a compact interval in the localization region satisfying $|N(J)|>0$.\\
  Define
  \begin{itemize}
  \item the probability density
    $\D\nu_J:=\frac{1}{|N(J)|}\nu(t)\car_J(t)$ where
    $\D\nu=\frac{dN}{dE}$ is the density of states of $H_\omega$;
  \item the point process $\D\tilde\Xi_J(\omega,t,\Lambda) =
    \sum_{E_n(\omega,\Lambda)\in J}
    \delta_{\nu(t)|\Lambda|[E_n(\omega,\Lambda)-t]}$.
  \end{itemize}
  Then, $\omega$-almost surely, the probability law of the point
  process $\tilde\Xi_J(\omega,\cdot,\Lambda)$ under the distribution
  $\D\nu_J(t)dt$ converges to the law of the Poisson point process on
  the real line with intensity $1$.
\end{Th}
\noindent We note that, in Theorem~\ref{thr:8}, we don't make any
regularity assumption on $N$ except for the Wegner estimate. This
enables us to remove the regularity condition imposed on the density
of states $\nu$ in the proof of the almost sure convergence of the
level spacings statistics given in~\cite{Ge-Kl:10}.  Thus, we prove
\begin{Th}
  \label{thr:7}
  Assume (IAD), (W), (M) and (Loc) hold.  Pick $J\subset I$ a compact
  interval in the localization region such that $|N(J)|>0$. Let
  $N(J,\omega,\Lambda)$ be the random number of eigenvalues of
  $H_\omega(\Lambda)$ is $J$. Define the eigenvalue or level spacings
  as
  \begin{equation*}
    \forall 1\leq j\leq N(J,\omega,\Lambda),\quad 
    \delta_J E_j(\omega,\Lambda)= \frac{|N(J)|}{|J|} |\Lambda|
    (E_{j+1}(\omega,\Lambda)-E_j(\omega,\Lambda))\geq 0
  \end{equation*}
  and the empirical distribution of these spacings to be the random
  numbers, for $x\geq0$
  \begin{equation*}
    DLS(x;J,\omega,\Lambda)=\frac{\#\{j;\ E_j(\omega,\Lambda)\in
      J,\ \delta_J E_j(\omega,\Lambda)\geq x\}
    }{N(J,\omega,\Lambda)}.
  \end{equation*}
  Then, $\omega$-almost surely, as $|\Lambda|\to+\infty$,
  $DLS'(x;J,\omega,\Lambda)$ converges uniformly to the distribution
  $x\mapsto g_{\nu,J}(x)$ where $\D g_{\nu,J}(x)=
  \int_{J}e^{-\nu_J(\lambda)|J|x}\nu_J(\lambda) d\lambda$.
\end{Th}
\subsection{Microscopic energy intervals}
\label{sec:micr-energy-interv}
One can also prove a version of Theorem~\ref{thr:2} that is local in
energy. In this case, one needs that the weight the density of states
puts on the energy interval under consideration not be too small with
respect to the length of the energy interval (see the first condition
in~\eqref{eq:63}). One proves
\begin{Th}
  \label{thr:3}
  Assume (IAD), (W), (M) and (Loc) hold. Pick $E_0\in I$.\\
  Fix $(I_\Lambda)_\Lambda$ a decreasing sequence of intervals such
  that $\displaystyle\sup_{I_\Lambda}|x|\vers_{|\Lambda|\to+\infty}0$.
  Assume that, for some $\delta>0$ and
  $\tilde\rho\in(0,\rho/(1+(1+\rho)d))$ (recall that $\rho$ is defined
  in (M)), one has
  \begin{gather}
    \label{eq:63}
    N(E_0+I_\Lambda)\cdot|I_\Lambda|^{-1-\tilde\rho}\geq1,\quad
    |\Lambda|^{1-\delta}\cdot N(E_0+I_\Lambda) \vers_{|\Lambda|\to+\infty}
    +\infty\\
    \intertext{ and }
    \label{eq:64}
    \text{if }\ell'=o(L)\text{ then
    }\frac{N(E_0+I_{\Lambda_{L+\ell'}})}{N(E_0+I_{\Lambda_L})}
    \vers_{L\to+\infty}1.
  \end{gather}
  Then, $\omega$-almost surely, the probability law of the point
  process $\Xi_{E_0+I_\Lambda}(\omega,\cdot,\Lambda)$ under the
  uniform distribution $\car_{[0,1]}(t)dt$ converges to the law of the
  Poisson point process on the real line with intensity $1$.
\end{Th}
\noindent Note that the first condition in~\eqref{eq:63} requires that
the derivative of $N$ does not vanish too fast at $E_0$. As a
consequence of Theorem~\ref{thr:3}, using the results
of~\cite{MR2352280}, one shows that one has convergence of the
unfolded local level spacings distribution at any point of the almost
sure spectrum if one looks at ``large'' enough neighborhoods of the
point; here, ``large'' does not mean that the neighborhood needs to be
large: it merely needs not to shrink too fast to $0$
(see~\eqref{eq:63}).
\subsection{Results for the random Hamiltonian on the whole space}
\label{sec:results-full-random}
In our previous results, we considered the eigenvalues of the random
Hamiltonian restricted to a box. As in~\cite{Ge-Kl:10}, one can also
consider the operator $H_\omega$ on the whole space. Therefore, we
recall
\begin{Pro}[\cite{Ge-Kl:10}]
  \label{pro:1}
  Assume (IAD), (W) and (Loc). Fix $q>2d$. Then, there exists
  $\gamma>0$ such that, $\omega$-almost surely, there exists
  $C_\omega>1$, $\esp(C_\omega)<\infty$, such that
  \begin{enumerate}
  \item with probability $1$, if $E\in I\cap\sigma(H_\omega)$ and
    $\varphi$ is a normalized eigenfunction associated to $E$ then,
    for some $x(E,\omega)\in\R^d$ or $\Z^d$, a maximum of
    $x\mapsto\|\varphi\|_x$, for some $C_\omega>0$, one has, for
    $x\in\R^d$,
    \begin{equation*}
      \|\varphi\|_x\leq C_\omega (1+|x(E,\omega)|^2)^{q/2}e^{-\gamma
        |x-x(E,\omega)|^\xi}
    \end{equation*}
    where $\|\cdot\|_x$ is defined in~\eqref{eq:22}.\\
    Moreover, one has $\esp(C_\omega)<+\infty$.\\
    $x(E,\omega)$ is a center of localization for $E$ or $\varphi$.
  \item Pick $J\subset I$ such that $|N(J)|>0$. Let
    $N^f(J,\Lambda,\omega)$ denotes the number of eigenvalues of
    $H_\omega$ having a center of localization in $\Lambda$. Then,
    there exists $\beta>0$ such that, for $\Lambda$ sufficiently
    large, one has
    \begin{equation*}
      \left|\frac{N^f(J,\Lambda,\omega)}{|N(J)|\,|\Lambda|}-1\right|
      \leq\frac1{\log^\beta|\Lambda|}.
    \end{equation*}
  \end{enumerate}
\end{Pro}
\noindent In view of Proposition~\ref{pro:1}, $\omega$-almost surely,
there are only finitely many eigenvalues of $H_\omega$ in $J$ having a
localization center in $\Lambda_L$. Thus, we can enumerate these
eigenvalues as $E^f_1(\omega,\Lambda)\leq E^f_2(\omega,\Lambda)\leq
\cdots\leq E^f_N(\omega,\Lambda)$ where we repeat them according to
multiplicity. For $t\in[0,1]$, define the point process
$\Xi^f_J(\omega,t,\Lambda)$ by~\eqref{eq:6} and~\eqref{eq:7} for those
eigenvalues. As a corollary of Theorem~\ref{thr:2}, we obtain
\begin{Th}
  \label{thr:5}
  Assume (IAD), (W), (M) and (Loc) hold.
  Assume that $J\subset I$, the localization region, that $|N(J)|>0$.\\
  Then, $\omega$-almost surely, the probability law of the point
  process $\Xi^f_J(\omega,\cdot,\Lambda)$ under the uniform
  distribution $\car_{[0,1]}(t)dt$ converges to the law of the Poisson
  point process on the real line with intensity $1$.
\end{Th}
\noindent Theorem~\ref{thr:5} also admits an corresponding analogue
that is local in energy i.e. a counterpart of Theorem~\ref{thr:3}.
\subsection{Outline of the paper}
\label{sec:outline-paper}
Let us briefly outline the remaining parts of the paper. In
section~\ref{sec:spectr-rand}, we recall some results
from~\cite{Ge-Kl:10} that we build our analysis upon. The strategy of
the proof will be roughly to study the eigenvalues of the random
operator where the integrated density of states, $N(\cdot)$, takes
value close to $t$. Most of those eigenvalues, as in shown
in~\cite{Ge-Kl:10}, can be approximated by i.i.d. random variables the
distribution law of which is roughly uniform on $[0,1]$ when properly
renormalized. We then show that this approximation is accurate enough
to obtain the almost sure convergence announced in
Theorem~\ref{thr:2}.\\
Theorem~\ref{thr:3} is proved in the same way and we only make a few
remarks on this proof in section~\ref{sec:proof-theorem-1}.
Theorem~\ref{thr:5} is deduced from Theorem~\ref{thr:2} approximating
the eigenvalues of $H_\omega$ by those of $H_\omega(\Lambda)$ for
sufficiently large $\Lambda$; this is done in
section~\ref{sec:proof-theorem-2}.\\
Section~\ref{sec:proof-theorems} is devoted to the proof of
Theorems~\ref{thr:8} and~\ref{thr:7} . It relies on point process
techniques, in particular, on transformations of point processes (see
e.g.~\cite{MR1700749,MR2364939}).
\section{The spectrum of a random operator in the localized regime}
\label{sec:spectr-rand}
Let us now recall some results taken from~\cite{Ge-Kl:10} that will
use in our proofs.
\subsection{Distribution of the unfolded eigenvalues}
\label{sec:distr-eigenv-small}
We now describe the distribution of the unfolded eigenvalues for the
operator $H_\omega$ in a small cube. Pick $1\ll\ell'\ll\ell$. Consider
a cube $\Lambda=\Lambda_\ell$ centered at $0$ of side length
$\ell$. Pick an interval $I_\Lambda=[a_\Lambda,b_\Lambda]\subset I$
(i.e. $I_\Lambda$ is contained in the localization region) for $\ell$
sufficiently large.\\
Consider the following random variables:
\begin{itemize}
\item $X=X(\Lambda,I_\Lambda)=X(\Lambda,I_\Lambda,\ell')$ is the
  Bernoulli random variable
  \begin{equation*}
    X=\car_{H_\omega(\Lambda)\text{ has exactly one
        eigenvalue in }I_\Lambda\text{ with localization center in }
      \Lambda_{\ell-\ell'}}
  \end{equation*}
\item $\tilde E=\tilde E(\Lambda,I_\Lambda)$ is this eigenvalue
  conditioned on $X=1$.
\end{itemize}
Let $\tilde\Xi$ be the distribution function of $\tilde E$. We know
\begin{Le}[\cite{Ge-Kl:10}]
  \label{le:7}
  Assume (W), (M) and (Loc) hold.\\
  For $\kappa\in(0,1)$, one has
  \begin{equation}
    \label{eq:51}
    \left|\pro(X=1)-|N(I_\Lambda)||\Lambda|\right|\lesssim
    (|\Lambda||I_\Lambda|)^{1+\rho}+|N(I_\Lambda)||\Lambda|\ell'\ell^{-1}
      +|\Lambda|e^{-(\ell')^\kappa}
  \end{equation}
  where $N(E)$ denotes the integrated density of states of
  $H_\omega$.\\
  One has 
  \begin{equation*}
    \left|(\tilde\Xi(x)-\tilde\Xi(y))\,P(X=1)\right|\lesssim
    |x-y||I_\Lambda||\Lambda|.
  \end{equation*}
  Moreover, setting $N(x,y,\Lambda):=[N(a_\Lambda+x|I_\Lambda|)-
  N(a_\Lambda+y|I_\Lambda|)]|\Lambda|$, one has
  \begin{multline}
    \label{eq:52}
    \left|(\tilde\Xi(x)-\tilde\Xi(y))\,P(X=1)-N(x,y,\Lambda)\right|\\
    \lesssim (|\Lambda||I_\Lambda|)^{1+\rho} +
    |N(x,y,\Lambda)|\ell'\ell^{-1}+|\Lambda|e^{-(\ell')^\kappa}.
  \end{multline}
\end{Le}
\noindent Estimates~\eqref{eq:51} and~\eqref{eq:52} are of interest
mainly if their right hand side, which is to be understood as an error
term, is smaller than the main term. In~\eqref{eq:51}, the main
restriction comes from the requirement that $N(I_\Lambda)|\Lambda|\gg
(|\Lambda||I_\Lambda|)^{1+\rho}$ which is essentially a requirement
that $N(I_\Lambda)$ should not be too small with respect to
$|I_\Lambda|$. Lemma~\ref{le:7} will be used in conjunction with
Theorems~\ref{thr:vbig1}. The cube $\Lambda$ in Lemma~\ref{le:7} will
be the cube $\Lambda_\ell$ in Theorem~\ref{thr:vbig1}. Therefore, the
requirements induced by the other two terms are less restrictive. The
second term is an error term if $\ell'\ll\ell$ which is guaranteed by
assumption; this induces no new requirement. The third term in the
right hand side of~\eqref{eq:51} being small compared to
$|N(I_\Lambda)||\Lambda|$ requires that $|N(I_\Lambda)||\Lambda|\gg
\ell^d e^{-(\ell')^\kappa}$. This links the size of the cube
$\Lambda=\Lambda_\ell$ where we apply Lemma~\ref{le:7} to the size of
$|N(I_\Lambda)|$. The right choice for $\ell$ (that will become clear
from Theorem~\ref{thr:vbig1} stated below) is $\ell\asymp
|N(I_\Lambda)|^{-\nu}$. In our application, we will pick
$\ell'\asymp(\log\ell)^{1/\xi}$ for some $\xi\in(0,1)$ (coming from
the localization estimate (Loc)); so taking $\kappa>\xi$ ensures that
the third term in the right hand side of~\eqref{eq:51} is small
compared to $|N(I_\Lambda)||\Lambda|$. For further details, we refer
to the comments following the statement of Theorem~\ref{thr:vbig1} and
section~\ref{sec:reduct-local-behav} for
details.\\
In~\eqref{eq:52}, the main restriction comes from the requirement that
$N(x,y,\Lambda)|\Lambda|\gg (|\Lambda||I_\Lambda|)^{1+\rho}$. This is
essentially a requirement on the size of $|x-y|$. It should not be too
small. On the other hand, we expect the spacing between the
eigenvalues of $H_\omega(\Lambda_L)$ to be of size $|\Lambda_L|^{-1}$
(we keep the notations of Theorem~\ref{thr:vbig1} and recall that the
cube $\Lambda$ in Lemma~\ref{le:7} will be the cube $\Lambda_\ell$ in
Theorem~\ref{thr:vbig1}, hence, a cube much smaller that
$\Lambda_L$). So to distinguish between the eigenvalues, one needs to
be able to know $\tilde\Xi$ up to resolution
$|x-y||I_\Lambda|\sim|\Lambda_L|^{-1}$. This will force us to use
Lemma~\ref{le:7} on intervals $I_\Lambda$ such that
$|N(I_\Lambda)|\asymp|\Lambda|^{-\alpha}$ for some $\alpha\in(0,1)$
close to $1$ (see the discussion following Theorem~\ref{thr:vbig1} and
section~\ref{sec:reduct-local-behav}). Moreover, the approximation of
$\tilde\Xi(x)-\tilde\Xi(y)$ by $N(x,y,\Lambda)/P(X=1)$ will be good
if $|x-y|\gg(|\Lambda_L||I_\Lambda|)^{-1}\asymp |\Lambda_L|^{-\beta}$
for some $\beta>0$.
\subsection{I.I.D approximations to the eigenvalues}
\label{sec:i.i.d-appr-eigenv}
The second ingredient of our proof is a description of most of the
eigenvalues of $H_\omega(\Lambda)$ in some small interval, say,
$I_\Lambda$ in terms of i.i.d. random variables. These random
variables are the eigenvalues of the restrictions of
$H_\omega(\Lambda)$ to much smaller disjoint cubes, the distribution
of which we computed in Lemma~\ref{le:7}. This description of the
eigenvalues of $H_\omega(\Lambda)$ holds with a probability close to
$1$.
\subsubsection{Localization estimates and localization centers}
\label{sec:local-estim-local}
We first recall a result of~\cite{Ge-Kl:10} defining and describing
localization centers, namely,
\begin{Le}[\cite{Ge-Kl:10}]
  \label{le:14}
  Under assumptions (W) and (Loc), for any $p>0$ and $\xi\in(0,1)$,
  there exists $q>0$ such that, for $L\geq1$ large enough, with
  probability larger than $1-L^{-p}$, if
  \begin{enumerate}
  \item $\varphi_{n,\omega}$ is a normalized eigenvector of
    $H_{\omega}(\Lambda_L)$ associated to $E_{n,\omega}\in I$,
  \item $x_n(\omega)\in \Lambda_L$ is a maximum of
    $x\mapsto\|\varphi_{n,\omega}\|_x$ in $\Lambda_L$,
  \end{enumerate}
  then, for $x\in\Lambda_L$, one has
  \begin{equation*}
    \|\varphi_{n,\omega}\|_x\leq L^q e^{-|x-x_n(\omega)|^\xi}
  \end{equation*}
  where $\|\cdot\|_x$ is defined in~\eqref{eq:22}.\\
  Define $\D C(\varphi)=\{x\in\Lambda;\
  \|\varphi\|_x=\max_{\gamma\in\Lambda} \|\varphi\|_{\gamma}\}$ to be
  the set of localization centers for $\varphi$. Then, the diameter of
  $C(\varphi_j(\omega,\Lambda))$ is less than
  $C_q(\log|\Lambda|)^{1/\xi}$.
\end{Le}
\noindent We define localization centers in a unique way by ordering
the set $C(\varphi)$ lexicographically and take the supremum.
\subsubsection{An approximation theorem for eigenvalues}
\label{sec:an-appr-theor}
Pick $\xi\in(0,1)$, $R>1$ large and $\rho'\in(0,\rho)$ where $\rho$ is
defined in (M). For a cube $\Lambda$, consider an interval
$I_\Lambda=[a_\Lambda,b_\Lambda]\subset I$. Set $\ell'_\Lambda= (R\log
|\Lambda|)^{\frac 1\xi}$. We say that the sequence
$(I_\Lambda)_\Lambda$ is $(\xi,R,\rho')$-admissible if, for any
$\Lambda$, one has
\begin{equation}
  \label{eq:69}
  |\Lambda| |N(I_\Lambda)|\geq1,\quad
  |N(I_\Lambda)||I_\Lambda|^{-(1+\rho')}\geq1, \quad 
  |N(I_\Lambda)|^{\frac 1{1+\rho'}} (\ell'_\Lambda)^d\leq 1.
\end{equation}
One has
\begin{Th}[\cite{Ge-Kl:10}]
  \label{thr:vbig1}
  Assume (IAD), (W), (M) and (Loc) hold. Let $\Lambda=\Lambda_L$ be
  the cube of center $0$ and side length $L$.\\
  Pick $\rho'\in[0,\rho/(1+(\rho+1)d))$ where $\rho$ is defined in
  (M).  Pick a sequence of intervals that is
  $(\xi,R,\rho')$-admissible, say, $(I_\Lambda)_\Lambda$ such that
  $\ell'_\Lambda\ll\tilde\ell_\Lambda\ll L$ and $|N(I_\Lambda)|^{\frac
    1{1+\rho'}}\tilde\ell_\Lambda^d\to0$ as
  $|\Lambda|\to\infty$.\\
  For any $p>0$, for $L$ sufficiently large (depending only on
  $(\xi,R,\rho',p)$ but not on the admissible sequence of intervals),
  there exists
  \begin{itemize}
  \item a decomposition of $\Lambda_L$ into disjoint cubes of the form
    $\Lambda_{\ell_\Lambda}(\gamma_j):=\gamma_j+[0,\ell_\Lambda]^d$,
    where
    $\D\ell_\Lambda=\tilde\ell_\Lambda(1+\mathcal{O}(\tilde\ell_\Lambda
    / |\Lambda_L|))=\tilde\ell_\Lambda(1+o(1))$ such that
    \begin{itemize}
    \item $\cup_j\Lambda_{\ell_\Lambda}(\gamma_j)\subset\Lambda_L$,
    \item $\dist
      (\Lambda_{\ell_\Lambda}(\gamma_j),\Lambda_{\ell_\Lambda}(\gamma_k))\ge
      \ell'_\Lambda$ if $j\not=k$,
    \item $\dist
      (\Lambda_{\ell_\Lambda}(\gamma_j),\partial\Lambda_L)\ge
      \ell'_\Lambda$,
    \item
      $|\Lambda_L\setminus\cup_j\Lambda_{\ell_\Lambda}(\gamma_j)|\lesssim
      | \Lambda_L| \ell'_\Lambda/\ell_\Lambda$,
    \end{itemize}
  \item a set of configurations $\mathcal{Z}_\Lambda$ such that
    \begin{itemize}
    \item $\mathcal{Z}_\Lambda$ is large, namely,
      \begin{equation}
        \label{eq:89}
        \pro(\mathcal{Z}_\Lambda)\geq 1 -\frac12|\Lambda|^{-p}- 
        \exp\left(-c|I_\Lambda|^{1+\rho}|\Lambda|
          \ell_\Lambda^{d\rho} \right)\\-\exp\left( -c
          |\Lambda||I_\Lambda|\ell'_\Lambda\ell_\Lambda^{-1}\right)
      \end{equation}
    \end{itemize}
  \end{itemize}
  so that
  \begin{itemize}
  \item for $\omega\in\mathcal{Z}_\Lambda$, there exists at least
    $\D\frac{|\Lambda|}{\ell_\Lambda^d}\left(1+
      O\left(|N(I_\Lambda)|^{1/(1+\rho')} \ell_\Lambda^d\right)\right)$
    disjoint boxes $\Lambda_{\ell_\Lambda}(\gamma_j)$ satisfying the
    properties:
    \begin{enumerate}
    \item the Hamiltonian $H_\omega(\Lambda_{\ell_\Lambda}(\gamma_j))$
      has at most one eigenvalue in $I_\Lambda $, say,
      $E_n(\omega,\Lambda_{\ell_\Lambda}(\gamma_j))$;
    \item $\Lambda_{\ell_\Lambda}(\gamma_j)$ contains at most one
      center of localization, say $x_{k_j}(\omega,L)$, of an
      eigenvalue of $H_\omega(\Lambda)$ in $I_\Lambda $, say
      $E_{k_j}(\omega,\Lambda)$;
    \item $\Lambda_{\ell_\Lambda}(\gamma_j)$ contains a center
      $x_{k_j}(\omega,\Lambda)$ if and only if
      $\sigma(H_\omega(\Lambda_{\ell_\Lambda}(\gamma_j)))\cap
      I_\Lambda\not=\emptyset$; in which case, one has
      \begin{equation}
        \label{eq:23}
        |E_{k_j}(\omega,\Lambda)-E_n(\omega,\Lambda_{\ell_\Lambda}(\gamma_j))| \leq  
        |\Lambda|^{-R}\text{ and }\mathrm{dist}(x_{k_j}(\omega,L),
        \Lambda_L \setminus \Lambda_{\ell_\Lambda}(\gamma_j))\geq \ell'_\Lambda
      \end{equation}
      where we recall that $\ell'_\Lambda= (R\log |\Lambda|)^{\frac
        1\xi}$;
    \end{enumerate}
  \item the number of eigenvalues of $H_\omega(\Lambda)$ that are not
    described above is bounded by
    \begin{equation}
      \label{eq:16}
      C |N(I_\Lambda)||\Lambda|
      \left(|N(I_\Lambda)|^{\frac{\rho-\rho'}{1+\rho'} }
        \ell_\Lambda^{d(1+\rho)}+|N(I_\Lambda)|^{-\frac{\rho'}
          {1+\rho'}}
        (\ell'_\Lambda)^{d+1}\ell_\Lambda^{-1}\right);
    \end{equation}
    this number is $o(|N(I_\Lambda)| |\Lambda|)$ provided
    \begin{equation}
      \label{condell} 
      |N(I_\Lambda)|^{-\frac{\rho'}{1+\rho'}}
      (\ell'_\Lambda)^{d+1} \ll \ell_\Lambda \ll
      |N(I_\Lambda)|^{- \frac{\rho-\rho'}{d(1+\rho)(1+\rho')}} .
    \end{equation}
  \end{itemize}
\end{Th}
\noindent We note that the assumptions on $(I_\Lambda)_\Lambda$ in
Theorem~\ref{thr:vbig1} imply that $|I_\Lambda|\to0$ and
$|N(I_\Lambda)|$ must go to $0$ faster than logarithmically in
$|\Lambda|$ (see the right hand side of~\eqref{condell}).\\
Let us now briefly explain how the lengthscale $\ell=\ell_\Lambda$
will be chosen in our analysis (see
section~\ref{sec:reduct-local-behav}). We will use
Theorem~\ref{thr:vbig1} on intervals $I_\Lambda$ such that
$|N(I_\Lambda)||\asymp|\Lambda|^{-\alpha}$ (for some $\alpha\in(0,1)$
close to $1$) and set $\ell_\Lambda\asymp |N(I_\Lambda)|^{-\nu}$ for
some $\nu\in(0,1)$. Thus, $\log\ell'_\Lambda\ll\log\ell_\Lambda$ and
checking of the validity of~\eqref{condell} reduces to checking that
$\frac{\rho'}{1+\rho'}< \frac{\rho-\rho'}{d(1+\rho)(1+\rho')}$ which
follow from the assumption $\rho'\in[0,\rho/(1+(\rho+1)d))$. The
exponent $\nu$ is then chosen so that
\begin{equation}
  \label{eq:25}
  \frac{\rho'}{1+\rho'}<\nu<
  \frac{\rho-\rho'}{d(1+\rho)(1+\rho')}.  
\end{equation}
Note that the right hand side inequality in~\eqref{condell} implies
that $|N(I_\Lambda)|^{\frac 1{1+\rho'}}\tilde\ell_\Lambda^d\to0$ as
$(\rho-\rho')/(1+\rho)<1$. With these choices, the bound~\eqref{eq:16}
then becomes $|N(I_\Lambda)||\Lambda|^{1-\beta}$ for some
$\beta>0$.\\
To conclude this section, we note that, when the length scales are
chosen as just indicated, one easily checks that the
estimate~\eqref{eq:89} becomes 
\begin{equation}
  \label{eq:11}  
  \pro(\mathcal{Z}_\Lambda)\geq 1 - |\Lambda|^{-p}
\end{equation}
\subsection{A large deviation principle for the eigenvalue counting
  function}
\label{sec:large-devi-princ}
Define the random numbers
\begin{equation}
  \label{eq:26}
  N(I_\Lambda,\Lambda,\omega):=\#\{j;\ E_j(\omega,\Lambda)\in
  I_\Lambda\}.
\end{equation}
Write $I_\Lambda=[a_\Lambda,b_\Lambda]$ and recall that
$|N(I_\Lambda)|=N(b_\Lambda)-N(a_\Lambda)$ where $N$ is the integrated
density of states. Using Theorem~\ref{thr:vbig1} and standard large
deviation estimates for i.i.d. random variables, one shows that
$N(I_\Lambda,\Lambda,\omega)$ satisfies a large deviation principle,
namely,
\begin{Th}
  \label{thr:4}
  Assume (IAD), (W), (M) and (Loc) hold. For any
  $\rho'\in(0,\rho/(1+(1+\rho)d))$ ($\rho$ is defined in
  Assumption~(M)) $\delta\in(0,1)$ and $\nu\in(0,1-\delta)$, there
  exists $\delta'>0$ such that, if $(I_\Lambda)_\Lambda$ is a sequence
  of compact intervals in the localization region $I$ satisfying
  \begin{enumerate}
  \item $|N(I_\Lambda)||\Lambda|^{\delta}\to0$ as
    $|\Lambda|\to+\infty$
  \item $|N(I_\Lambda)|\,|\Lambda|^{1-\nu}\to+\infty$ as
    $|\Lambda|\to+\infty$
  \item $|N(I_\Lambda)|\,|I_\Lambda|^{-1-\rho'}\to+\infty$ as
    $|\Lambda|\to+\infty$,
  \end{enumerate}
  then, for any $p>0$, for $|\Lambda|$ sufficiently large (depending
  on $\rho'$ and $\nu$ but not on the specific sequence
  $(I_\Lambda)_\Lambda$), one has
  \begin{equation}
    \label{eq:12}
    \pro\left(\left|N(I_\Lambda,\Lambda,\omega)-|N(I_\Lambda)||\Lambda|
      \right|\geq |N(I_\Lambda)||\Lambda|^{1-\delta'}\right)\leq
    |\Lambda|^{-p}.
  \end{equation}
\end{Th}
\noindent This result is essentially Theorem 1.8 in~\cite{Ge-Kl:10};
the only change is a change of scale for $|N(I_\Lambda)|$ in terms of
$|\Lambda|$ (see point (1)). Up to this minor difference, the proofs
of the two results are the same.\\
Assume that, for $J$, an interval in the region of localization $I$,
one has the lower bound $|N(x)-N(y)| \gtrsim|x-y|^{1+\rho'}$ for
$(x,y)\in I^2$ and some $\rho'\in(0,\rho/(1+(1+\rho)d))$. Then, as
$K\mapsto |N(K)|$ is a measure, thus, additive, for $K\subset J$ the
region of localization, one may split $K$ into intervals $(K_k)_k$
such that $|N(K_k)|\asymp|\Lambda|^{-\delta}$, and sum the estimates
given by Theorem~\ref{thr:4} on each $K_k$ to obtain that
\begin{equation*}
  \pro\left(\left|N(K,\Lambda,\omega)-|N(K)||\Lambda| \right|\geq
    |N(K)||\Lambda|^{1-\delta'}\right)\lesssim
  |\Lambda|^{-p}.
\end{equation*}
Though we will not need it, this gives an interesting large deviation
estimate for intervals of macroscopic size.
\section{The proofs of Theorems~\ref{thr:2},~\ref{thr:5}
  and~\ref{thr:3}}
\label{sec:proof-theorem}
\noindent We first prove Theorem~\ref{thr:2}. Theorem~\ref{thr:5} is
then a immediate consequence of Theorem~\ref{thr:2} and the fact that
most of the eigenvalues of $H_\omega(\Lambda)$ and those of $H_\omega$
having center of localization in $\Lambda$ differ at most by $L^{-p}$
for any $p$ and $L$ sufficiently large (see
section~\ref{sec:proof-theorem-2}). Theorem~\ref{thr:3} is proved in
the same way as Theorem~\ref{thr:2} in
section~\ref{sec:proof-theorem-2}; we skip most of the details of
this proof. \\
We shall use the following standard notations: $a\lesssim b$ means
there exists $c<\infty$ so that $a\le cb$; $\langle x\rangle =
(1+|x|^2)^\frac12$. We write $a\asymp b$ when $a\lesssim b$ and $b
\lesssim a$.\\
From now on, to simplify notations, we write $N$ instead of $N_J$ so
that the density of states increases from $0$ to $1$ on $J$. We also
write $\Xi$  instead of $\Xi_J$\\
For $\varphi:\ \R\to\R$ continuous and compactly supported, set
\begin{equation}
  \label{eq:4}
  {\mathcal L}_{\omega,\Lambda}(\varphi):={\mathcal L}_{\omega,J,\Lambda}:=
  \int_0^1e^{-\langle\Xi(\omega,t,\Lambda),\varphi\rangle}dt
\end{equation}
and
\begin{equation}
  \label{eq:5}
  \langle\Xi(\omega,t,\Lambda),\varphi\rangle:=
  \sum_{E_n(\omega,\Lambda)\in J}
  \varphi(|\Lambda|[N(E_n(\omega,\Lambda))-t])
\end{equation}
To prove Theorems~\ref{thr:2} and~\ref{thr:3}, it suffices
(see~\cite{Mi:11}) to prove
\begin{Th}
  \label{thr:6}
  For $\varphi:\ \R\to\R^+$ continuously differentiable and compactly
  supported, $\omega$-almost surely,
  \begin{equation}
    \label{eq:3}
    {\mathcal L}_{\omega,\Lambda}(\varphi)\vers_{|\Lambda|\to+\infty}
    \exp\left(-\int_{-\infty}^{+\infty}\left(1-e^{-\varphi(x)}\right)dx\right).
  \end{equation}
\end{Th}
\noindent Then, a standard dense subclass argument shows that the
limit~\eqref{eq:3} holds for compactly supported, continuous, non
negative functions. This completes the proof of Theorem~\ref{thr:2}.
\subsection{The proof of Theorem~\ref{thr:6}}
\label{sec:proof-theor-refthr:6}
The integrated density of states $N$ is non decreasing. By assumption
(W), it is Lipschitz continuous. One can partition
$[0,1]=\cup_{n\in\mathcal{N}} I_n$ where $\mathcal{N}$ is at most
countable and $(I_n)_{n\in\mathcal{N}}$ are intervals such that either
\begin{itemize}
\item $I_n$ is open and $N$ is strictly increasing on the open
  interval $N^{-1}(I_n)$; we then say that $n\in\mathcal{N}^+$;
\item $I_n$ reduces to a single point and $N$ is constant on the
  closed interval $N^{-1}(I_n)$; we then say that $n\in\mathcal{N}^0$.
\end{itemize}
We prove
\begin{Le}
  \label{le:1}
  For the limit~\eqref{eq:3} to hold $\omega$-almost surely, it
  suffices that, for any $n\in\mathcal{N}^+$, for
  $\varphi:\,\R\to\R^+$ continuously differentiable and compactly
  supported, $\omega$-almost surely, one has
  \begin{equation}
    \label{eq:67}
    \left|{\mathcal L}_{\omega,I_n,\Lambda}(\varphi)-
      \exp\left(-\int_{-\infty}^{+\infty}
        \left(1-e^{-\varphi(x)}\right)dx\right)\right|
    \vers_{|\Lambda|\to+\infty}0.
  \end{equation}
\end{Le}
\begin{proof}
  As for $n\in\mathcal{N}^0$, $I_n$ is a single point, one computes
  \begin{equation}
    \label{eq:8}
    {\mathcal L}_{\omega,\Lambda}(\varphi)=\sum_{n\in\mathcal{N}^+} 
    \int_{I_n}e^{-\langle\Xi(\omega,t,\Lambda),\varphi\rangle}dt.
  \end{equation}
  Assume $J=[a,b]$. Fix $t\in I_m=(N(a_m),N(b_m))$ for some
  $m\in\mathcal{N}^+$. For $m\in\mathcal{N}^0$, $N$ is constant equal
  to, say, $N_m$ on $I_m$. Assume that $\varphi$ has its support in
  $(-R,R)$. Then, for $|\Lambda|$ large (depending only on $R$), one
  computes
  \begin{equation*}
    \begin{split}
      \langle\Xi(\omega,t,\Lambda),\varphi\rangle&=
      \sum_{m\in\mathcal{N}^0} \#\{E_n(\omega,\Lambda)\in I_m\}
      \varphi(|\Lambda|[N_n-t)])\\&\hskip1cm+
      \sum_{m\in\mathcal{N}^+}\sum_{E_n(\omega,\Lambda)\in I_m}
      \varphi(|\Lambda|[N(E_n(\omega,\Lambda))-t)])\\
      &= \sum_{E_n(\omega,\Lambda)\in I_m}
      \varphi(|\Lambda|[N(E_n(\omega,\Lambda))-t)])\\
      &= \sum_{E_n(\omega,\Lambda)\in I_m}
      \varphi(|N(I_m)||\Lambda|[N_{I_n}(E_n(\omega,\Lambda))
      -(t-N(a_m))/|N(I_m)|)])\\
      &=
      \langle\Xi_{I_m}(\omega,(t-N(a_m))/|N(I_m)|,\Lambda),\varphi\rangle
    \end{split}
  \end{equation*}
  On the other hand
  \begin{equation*}
    \int_{N(a_m)}^{N(b_m)}e^{-\langle\Xi_{I_m}(\omega,(t-N(a_m))/|N(I_m)|,\Lambda),
      \varphi\rangle}dt=
    |N(I_m)|\int_0^1e^{-\langle\Xi_{I_m}(\omega,t,\Lambda), \varphi\rangle}dt.
  \end{equation*}
  Recall that, as the measure defined by $N$ is absolutely continuous
  with respect to the Lebesgue measure, we have
  \begin{equation*}
    \sum_{n\in\mathcal{N}^+}|N(I_m)|=|N(J)|=1.
  \end{equation*}
  Thus, by Lebesgue's dominated convergence theorem, as
  $\mathcal{N}^+$ is at most countable, we get that, if the necessary
  condition given in Lemma~\ref{le:1} is satisfied, then
  $\omega$-almost surely, we get
  \begin{equation*}
    \lim_{|\Lambda|\to+\infty}{\mathcal L}_{\omega,\Lambda}(\varphi)=
    \sum_{n\in\mathcal{N}^+}|N(I_n)|\lim_{|\Lambda|\to+\infty}
    {\mathcal L}_{\omega,I_m,\Lambda}(\varphi).
  \end{equation*}
  Thus, we have proved Lemma~\ref{le:1}.
\end{proof}
\noindent From now on, we assume that $N$ is a strictly increasing
one-to-one mapping from $J$ to $[0,1]$ and prove Theorem~\ref{thr:6}
under this additional assumption.\vskip.1cm
\noindent Therefore, we first bring ourselves back to proving a
similar result for ``local'' eigenvalues i.e. eigenvalues of
restrictions of $H_\omega(\Lambda)$ to cubes much smaller than
$\Lambda$ that lie inside small intervals i.e. much smaller than
$J$. The ``local'' eigenvalues are those described by points (1), (2),
(3) of Theorem~\ref{thr:vbig1}. Using Lemma~\ref{le:7} then
essentially brings ourselves back to the case of i.i.d. random
variables uniformly distributed on $[0,1]$.\\
Theorem~\ref{thr:vbig1} does not give control on all the eigenvalues.
To control the integral~\eqref{eq:4}, this is not necessary: a good
control of most of the eigenvalues is sufficient as Lemma~\ref{le:4}
below shows. Theorem~\ref{thr:4}, which is a corollary of
Theorem~\ref{thr:vbig1}, is used to obtain good bounds on the number
of controlled eigenvalues in the sense of Lemma~\ref{le:4}.
\subsection{Reduction to the study of local eigenvalues}
\label{sec:reduct-local-behav}
Assume we are in the setting of Theorem~\ref{thr:2} and that $N$ is as
above i.e. $N$ is a strictly increasing Lipschitz continuous function
from $J$ to $[0,1]$. Recall that $\nu$ is its derivative, the density
of states.\\
To obtain our results, we will use Theorem~\ref{thr:vbig1} and
Lemma~\ref{le:7}. Therefore, we split the interval $I$ into small
intervals and choose the length scale $\ell=\ell_\Lambda$ so that we
can apply both Theorem~\ref{thr:vbig1} and Lemma~\ref{le:7} to these
intervals. We now explain how this choice is done.\\
Recall that $\rho$ is defined in (M) and pick
$\rho'\in(0,\rho/(1+(\rho+1)d))$. The computations done
in~\cite[section 4.3.1]{Ge-Kl:10} show that for any $\alpha\in(0,1)$
and $\nu\in(0,1/d)$ such that
\begin{equation}
  \label{eq:87}
  0<1-\frac1\alpha+\frac{\rho-\rho'}{1+\rho'}-d\nu\rho,
\end{equation}
for $I_\Lambda$ (in the localization region) and $\ell=\ell_\Lambda$
such that
\begin{equation}
  \label{eq:58}
  |N(I_\Lambda)|\asymp|\Lambda|^{-\alpha} \quad\text{and}\quad 
  \ell_\Lambda\asymp|N(I_\Lambda)|^{-\nu}
\end{equation}
if, in addition $I_\Lambda$ satisfies
\begin{equation}
  \label{eq:86}
  |N(I_\Lambda)|\geq |I_\Lambda|^{1+\rho'},
\end{equation}
we can apply Theorem~\ref{thr:vbig1} and Lemma~\ref{le:7} to
$I_\Lambda$ and
\begin{itemize}
\item \eqref{eq:69} and~\eqref{condell} are satisfied;
\item the right hand side in~\eqref{eq:12} is $o(|N(I_\Lambda)||\Lambda|)$.
\end{itemize}
From now on we fix $\alpha\in(0,1)$, $\nu\in(0,1/d)$ and $\beta>0$
such that~\eqref{eq:87} be satisfied and, for later purposes, such
that
\begin{equation}
  \label{eq:40}
  \frac1\alpha-1<\frac{\rho'}{1+\rho'}
\end{equation}
Partition $J=[a,b]$ into disjoint intervals $(J_{j,\Lambda})_{1\leq
  j\leq j_\Lambda}$ of weight $|N(J_{j,\Lambda})|=|\Lambda|^{-\alpha}$
so that $j_\Lambda=|\Lambda|^{\alpha}$.\\
Define the sets
\begin{equation}
  \label{eq:66}
  B=\left\{1\leq j\leq j_\Lambda;\ |N(J_{j,\Lambda})|\leq
    |J_{j,\Lambda}|^{1+\rho'}\right\}\quad\text{and}\quad
  G=\{1,\cdots,j_\Lambda\}\setminus B.
\end{equation}
The set $B$ is the set of ``bad'' indices $j$ for which the interval
$J_{j,\Lambda}$ does not satisfy the assumptions of
Theorem~\ref{thr:vbig1}, more precisely, does not satisfy the second
condition in~\eqref{eq:86}. \vskip.1cm\noindent
For $j\in B$, one has
\begin{equation*}
  |J_{j,\Lambda}|\geq
  |N(J_{j,\Lambda})|^{1/(1+\rho')}=|\Lambda|^{-\alpha/(1+\rho')}.
\end{equation*}
Thus, one gets
\begin{equation}
  \label{eq:88}
  \#B\leq |\Lambda|^{\alpha/(1+\rho')}
\end{equation}
Fix $\alpha'\in(\alpha,\min[1,\alpha(1+2\rho')/(1+\rho')])$. For $j\in
G$, write $J_{j,\Lambda}=[a_\Lambda,b_\Lambda)$ and define
\begin{equation*}
  K_{j,\Lambda}:= [a'_\Lambda,b'_\Lambda]\subset
  J_{j,\Lambda}\text{ where }
  \begin{cases}
    a'_\Lambda=\inf\left\{a\geq a_\Lambda;
      N(a)-N(a_\Lambda)\geq|\Lambda|^{-\alpha'}\right\},\\
    b'_\Lambda=\sup\left\{b\leq b_\Lambda;
      N(b_\Lambda)-N(b)\geq|\Lambda|^{-\alpha'}\right\}.
  \end{cases}
\end{equation*}
that is, $K_{j,\Lambda}$ is the interval $J_{j,\Lambda}$ where small
neighborhoods of the endpoints have been remove.\\
Thus, our construction yields that
\begin{enumerate}
\item the total density of states of the set we have remove is bounded
  by
  \begin{equation}
    \label{eq:9}
    \sum_{j\in B}|N(J_{j,\Lambda})|+\sum_{j\in G}
    N(J_{j,\Lambda}\setminus K_{j,\Lambda}) \lesssim
    |\Lambda|^{-\alpha+\alpha/(1+\rho')}
    +|\Lambda|^{-\alpha'+\alpha}\lesssim
    |\Lambda|^{-\alpha\rho'/(1+\rho')};
  \end{equation}
\item for $j\in G$, $t\in N(K_{j,\Lambda})$ and $E\in J_{j',\Lambda}$
  for $j'\not=j$, one has
  \begin{equation*}
    |\Lambda||N(E)-t|\gtrsim|\Lambda|^{1-\alpha'}.
  \end{equation*}
\end{enumerate}
Note that one has
\begin{equation}
  \label{eq:74}
  |N(J)|=\sum_{j\in G}|N(J_{j,\Lambda})|+\sum_{j\in B}|N(J_{j,\Lambda})|=
  \sum_{j\in G}|N(K_{j,\Lambda})|+
  O\left(|\Lambda|^{-\alpha\rho'/(1+\rho')}\right).
\end{equation}
Recall~\eqref{eq:5}. Thus, for $\Lambda$ sufficiently large, by point
(1) above, as $\varphi$ is non negative, one has
\begin{equation*}
  \begin{split}
    \int_0^1e^{-\langle\Xi(\omega,t,\Lambda),\varphi\rangle}dt&=
    \sum_{1\leq j\leq G} \int_{N(K_{j,\Lambda})}
    e^{-\langle\Xi(\omega,t,\Lambda),\varphi\rangle}dt+
    O\left(|\Lambda|^{-\alpha\rho'/(1+\rho')}\right)\\
    &=\sum_{1\leq j\leq G} \int_{N(K_{j,\Lambda})}
    e^{-\langle\Xi_j(\omega,t,\Lambda),\varphi\rangle}dt+
    O\left(|\Lambda|^{-\alpha\rho'/(1+\rho')}\right)
  \end{split}
\end{equation*}
where, as $\varphi$ is compactly supported, by point (2) above, one
has
\begin{equation*}
  \langle\Xi_j(\omega,t,\Lambda),\varphi\rangle=
  \sum_{E_n(\omega,\Lambda)\in J_{j,\Lambda}}
  \varphi(|\Lambda|[N(E_n(\omega,\Lambda))-t]).
\end{equation*}
Point (1) and~\eqref{eq:74} then yield
\begin{equation}
  \label{eq:73}
  \begin{split}
    \int_0^1e^{-\langle\Xi(\omega,t,\Lambda),\varphi\rangle}dt
    &=\sum_{j\in G} \int_{N(J_{j,\Lambda})}
    e^{-\langle\Xi_j(\omega,t,\Lambda),\varphi\rangle}dt+
    O\left(|\Lambda|^{-\alpha\rho'/(1+\rho')}\right) \\&=\sum_{j\in
      G}|N(J_{j,\Lambda})| \int_0^1
    e^{-\langle\Xi_{J_{j,\Lambda}}(\omega,t,\Lambda),\varphi\rangle}dt+
    O\left(|\Lambda|^{-\alpha\rho'/(1+\rho')}\right)
  \end{split}
\end{equation}
where $\Xi_{J_{j,\Lambda}}(\omega,t,\Lambda)$ is defined
by~\eqref{eq:6} for $J=J_{j,\Lambda}$. Thus, following the proof of
Lemma~\ref{le:1}, the limit~\eqref{eq:67} will hold $\omega$-almost
surely if we prove that, $\omega$ almost surely, one has
\begin{equation}
  \label{eq:72}
  \sup_{j\in G}\left|\int_0^1
    e^{-\langle\Xi_{J_{j,\Lambda}}(\omega,t,\Lambda),
      \varphi\rangle}dt-
    \exp\left(-\int_{-\infty}^{+\infty}
      \left(1-e^{-\varphi(x)}\right)dx\right)\right|
  \vers_{|\Lambda|\to+\infty}0.  
\end{equation}
Therefore, we first prove a weaker result, namely, almost sure
convergence along a subsequence, that is
\begin{Le}
  \label{le:2}
  Let $\Lambda_L$ be the cube of side length $L$ centered at $0$. Pick
  $(\alpha_L)_{L\geq1}$ any sequence valued in $[1/2,2]$ such that
  $\alpha_L\to1$ when $L\to+\infty$.\\
  There exists $\nu>0$ such that, for $\varphi:\,\R\to\R^+$
  continuously differentiable and compactly supported, $\omega$-almost
  surely, one has
  \begin{multline}
    \label{eq:71}
    \sup_{j\in G}\left|\int_0^1
      e^{-\langle\Xi_{J_{j,\Lambda_{L^\nu}}}(\omega,t,\Lambda_{L^\nu}),
        \varphi_{\alpha_L}\rangle}dt-
      \exp\left(-\int_{-\infty}^{+\infty}
        \left(1-e^{-\varphi(x)}\right)dx\right)\right|
    \vers_{L\to+\infty}0
  \end{multline}
  where, for $\alpha>0$, we have set,
  $\varphi_\alpha(\cdot)=\varphi(\alpha\,\cdot)$.
\end{Le}
\noindent Indeed, Lemma~\ref{le:2},~\eqref{eq:73} and~\eqref{eq:74}
clearly imply the claimed almost sure convergence on a subsequence;
more precisely, it implies that, for $(\alpha_L)_{L\geq1}$ a sequence
such that $\alpha_L\to1$ when $L\to+\infty$, $\omega$-almost surely,

\begin{equation}
  \label{eq:76}
  \left|{\mathcal L}_{\omega,\Lambda_{L^\nu}}(\varphi_{\alpha_L})-
    \exp\left(-\int_{-\infty}^{+\infty}\left(1-e^{-\varphi(x)}\right)dx\right)
  \right|\vers_{L\to+\infty}0.
\end{equation}
which is the claimed almost sure convergence on a subsequence for the
choice of sequence $\alpha_L=1$.\\
To obtain the almost sure convergence on the whole sequence, we use
\begin{Le}
  \label{le:8}
  For some $\beta>0$, for $\varphi:\,\R\to\R^+$ continuously
  differentiable and compactly supported, $\omega$-almost surely, for
  $L$ sufficiently large, one has
  \begin{equation}
    \label{eq:75}
    \sup_{L^\nu\leq L'\leq (L+1)^{\nu}}\left|
      {\mathcal L}_{\omega,\Lambda_{L'}}(\varphi)-
      {\mathcal L}_{\omega,\Lambda_{L^\nu}}(\varphi_{\alpha_{L'}})
    \right|\lesssim L^{-\beta}
  \end{equation}  
  where $\alpha_{L'}=|\Lambda_{L'}|/|\Lambda_{L^\nu}|$.
\end{Le}
\noindent As $\alpha_{L}\to1$ when $L\to+\infty$,
equation~\eqref{eq:72} and, thus, Theorem~\ref{thr:6}, are immediate
consequences of~\eqref{eq:76} and~\eqref{eq:75}.
\subsection{The proof of Lemma~\ref{le:2}}
\label{sec:proof-lemma}
The proof of Lemma~\ref{le:2} will consist in reducing the computation
of the limit~\eqref{eq:71} to the case of i.i.d. random variables that
have a distribution close to the uniform one. The number of these
random variables will be random as well but large; it is controlled by
Theorem~\ref{thr:4}.\\
We start with the statement and proof of a simple but useful result,
namely,
\begin{Le}
  \label{le:4}
  Pick a sequence of scale $(L_p)_{p\geq1}$ such that
  $L_p\to+\infty$. For $p\geq1$, consider two finite sequences
  $(x^p_n)_{1\leq n\leq N_p}$ and $(y^p_m)_{1\leq m\leq M_p}$ such
  that there exists $1\leq K_p\leq\inf(N_p,M_p)$ and sets
  $X_p\subset\{1,\cdots,N_p\}$ and $Y_p\subset\{1,\cdots,M_p\}$ s.t.
  \begin{enumerate}
  \item $\#X_p=\#Y_p=K_p$ and $[(N_p-K_p)+(M_p-K_p)]/L_p=:a_p\to0$,
  \item there exists a one-to-one map, say $\Psi_p:\ X_p\mapsto Y_p$
    such that, for $n\in X_p$, one has
    $|x^p_n-y^p_{\Psi_p(n)}|\leq\varepsilon_p/L_p$,
    $\varepsilon_p\in[0,1]$
  \end{enumerate}
  Fix $\alpha\in(0,1)$. Set $\D
  \Xi_p^x(t)=\sum_{n=1}^{N_p}\delta_{L_p[x^p_n-t]}$ and $\D
  \Xi_p^y(t)=\sum_{m=1}^{M_p}\delta_{L_p[y^p_m-t]}$. Then, for
  $p\geq1$, one has
  \begin{equation}
    \label{eq:80}
    \sup_{\varphi\in\mathcal{C}_{1,R}^+}
    \left|\int_0^1e^{-\langle\Xi_p^x(t),\varphi\rangle}dt-
      \int_0^1e^{-\langle\Xi_p^y(t),\varphi\rangle}dt\right|\leq 
    4a_p^{\alpha}+e^{R\,\varepsilon_p\,K_p}-1.
  \end{equation}
  where we have defined
  \begin{equation}
    \label{eq:65}
    \mathcal{C}_{1,R}^+=\left\{\varphi:\ \R\to\R^+;\
      \begin{aligned}
        \varphi\text{ is continuously differentiable
          s.t.}\\\text{supp}\,\varphi\subset(-R,R)\text{ and
        }\|\varphi\|_{\mathcal{C}^1}\leq R
      \end{aligned}
    \right\} 
  \end{equation}
\end{Le}
\begin{proof}[Proof of Lemma~\ref{le:4}]
  Let $\tilde X_p=\{1,\cdots,N_p\}\setminus X_p$ and $\tilde
  Y_p=\{1,\cdots,M_p\}\setminus Y_p$. For $(n,m)\in\tilde
  X_p\times\tilde Y_p$, define
  \begin{equation*}
    \begin{split}
      I^x_n=\begin{cases} x^p_n+a_p^{\alpha}[N_p-K_p]^{-1}[-1,1]\text{
          if }\tilde
        X_p\not=\emptyset\text{ i.e. }N_p-K_p\geq1,\\
        \emptyset\text{ if not};
      \end{cases}
      \\
      I^y_m=\begin{cases} y^p_m+a_p^{\alpha}[M_p-K_p]^{-1}[-1,1]\text{
          if }\tilde
        Y_p\not=\emptyset\text{ i.e. }M_p-K_p\geq1,\\
        \emptyset\text{ if not}.
      \end{cases}
    \end{split}
  \end{equation*}
  Then, by point (1) of our assumptions on the sequences $(x^p_n)_n$
  and $(y^p_m)_m$, one has
  \begin{gather}
    \label{eq:81}
    \begin{aligned}
      0&\leq\int_0^1e^{-\langle\Xi_p^x(t),\varphi\rangle}dt-
      \int_{[0,1]\setminus[(\cup_{n\in\tilde X_p}
        I^x_n)\cup(\cup_{m\in\tilde Y_p}
        I^y_m]}e^{-\langle\Xi_p^x(t),\varphi\rangle}dt\\&\leq
      (N_p-K_p)a_p^{\alpha}[N_p-K_p]^{-1}+
      (M_p-K_p)a_p^{\alpha}[M_p-K_p]^{-1}= 2a^\alpha_p
    \end{aligned}
    \intertext{and, similarly} \label{eq:82}
    0\leq\int_0^1e^{-\langle\Xi_p^y(t),\varphi\rangle}dt-
    \int_{[0,1]\setminus[(\cup_{n\in\tilde X_p}
      I^x_n)\cup(\cup_{m\in\tilde Y_p}
      I^y_m]}e^{-\langle\Xi_p^y(t),\varphi\rangle}dt\leq 2a^\alpha_p
  \end{gather}
  On the other hand, for $t\in[0,1]\setminus[(\cup_{n\in\tilde X_p}
  I^x_n)\cup(\cup_{m\in\tilde Y_p} I^y_m]$, one has
  \begin{equation*}
    L_p\,\text{dist}(t,\tilde X_p\cup\tilde Y_p)\geq 
    a_p^{\alpha}\,L_p\,\sup\left([N_p-K_p]^{-1},[N_p-K_p]^{-1}\right)\geq
    a_p^{\alpha-1}>R
  \end{equation*}
  for $p$ sufficiently large. Thus, for
  $t\in[0,1]\setminus[(\cup_{n\in\tilde X_p}
  I^x_n)\cup(\cup_{m\in\tilde Y_p} I^y_m]$ and
  $\varphi\in\mathcal{C}^+_{1,R}$ (see~\eqref{eq:65}), one has
  \begin{equation*}
    \langle\Xi_p^x(t),\varphi\rangle=\sum_{n\in
      X_p}\varphi(L_p[x^p_n-t])\quad\text{ and
    }\quad\langle\Xi_p^y,\varphi\rangle=\sum_{m\in 
      Y_p}\varphi(L_p[y^p_m-t]).
  \end{equation*}
  Now, by point (2) of our assumptions on the sequences $(x^p_n)_n$
  and $(y^p_m)_m$, one has
  \begin{equation}
    \label{eq:27}
    \sup_{\varphi\in\mathcal{C}_{1,R}^+}
    \sup_{\substack{t\in[0,1]\\t\not\in(\cup_{n\in\tilde X_p}
        I^x_n)\\t\not\in(\cup_{m\in\tilde Y_p} I^y_m)}}
    \left|\langle\Xi_p^x(t),\varphi\rangle-\langle\Xi_p^y(t),\varphi\rangle\right|
    \leq
    \varepsilon_p\,K_p\cdot\sup_{\varphi\in\mathcal{C}_{1,R}^+}
    \|\varphi'\|_\infty\leq R\,\varepsilon_p\,K_p.
  \end{equation}
  Hence, as $\varphi$ is non negative, we obtain
  \begin{equation}
    \label{eq:85}
    \sup_{\varphi\in\mathcal{C}_{1,R}^+}
    \left|\int_{[0,1]\setminus[(\cup_{n\in\tilde X_p}
        I^x_n)\cup(\cup_{n\in\tilde Y_p} I^y_n)]}
      \left(e^{-\langle\Xi_p^x(t),\varphi\rangle}-
        e^{-\langle\Xi_p^y(t),\varphi\rangle}\right)dt\right|
    \leq e^{R\,\varepsilon_p\,K_p}-1        
  \end{equation}
  Combining~\eqref{eq:81},~\eqref{eq:82} and~\eqref{eq:85} completes
  the proof of Lemma~\ref{le:4}.
\end{proof}
\begin{Rem}
  \label{rem:3}
  Lemma~\ref{le:4}, and, in particular, the error term coming
  from~\eqref{eq:27}, can be improved if one assumes that the points
  in the sequences are not too densely packed. This is the case in the
  applications we have in mind. Though we do not use it here, it may
  be useful to treat the case of long range correlated random
  potentials where the error estimates of the local approximations of
  eigenvalues given by Theorem~\ref{thr:vbig1} can not be that precise
  anymore.
\end{Rem}
\noindent Fix $j\in G$ (see~\eqref{eq:66}). Pick $R$ large in
Theorem~\ref{thr:vbig1}. The construction done in the beginning of
section~\ref{sec:reduct-local-behav} with the choice of scale
$\ell_\Lambda$ given by~\eqref{eq:58} implies that one can apply
\begin{itemize}
\item Theorem~\ref{thr:vbig1} to the energy interval
  $I_\Lambda:=J_{j,\Lambda}$ for $H_\omega(\Lambda_L)$, the small
  cubes being of side length $\ell_\Lambda$;
\item Theorem~\ref{thr:vbig1} and Lemma~\ref{le:7} to the energy interval
$I_\Lambda:=J_{j,\Lambda}$ and any of the cubes $\Lambda_\ell(\gamma)$
of the decompostion obtained in Theorem~\ref{thr:vbig1}.
\end{itemize}
Thus, we let $\mathcal{Z}^j_\Lambda$ be the set of configurations
$\omega$ defined by Theorem~\ref{thr:vbig1} for the energy interval
$I_\Lambda=J_{j,\Lambda}$. Then,~\eqref{eq:11} gives a lower bound on
$\pro(\mathcal{Z}^j_\Lambda)$ for any $p$ if $\Lambda$ is sufficiently
large (see the comment following Theorem~\ref{thr:vbig1}).\\
Let $\mathcal{N}^b_{\omega,j,\Lambda}$ be the set of indices $n$ of
the eigenvalues $(E_n(\omega,\Lambda))_n$ of $H_\omega(\Lambda)$ in
$J_{j,\Lambda}$ that are not described by (1)-(3) of
Theorem~\ref{thr:vbig1}. Let $\mathcal{N}^g_{\omega,j,\Lambda}$ be the
complementary set. Both sets are random. By~\eqref{condell} and our
choice of lengthscales (see teh comment following
Theorem~\ref{thr:vbig1}), the number of eigenvalues not described by
(1), (2) and (3) of Theorem~\ref{thr:vbig1}, say,
$N^b_{\omega,j,\Lambda}:=\#\mathcal{N}^b_{\omega,j,\Lambda}$ is
bounded by, for some $\beta>0$,
\begin{equation}
  \label{eq:18}
  N^b_{\omega,j,\Lambda}\leq |N(J_{j,\Lambda})||\Lambda|^{1-\beta}
\end{equation}
where as, by~\eqref{eq:12} in Theorem~\ref{thr:4}, the total number of
eigenvalue of $H_\omega(\Lambda)$ in $J_{j,\Lambda}$, say,
$N(J_{j,\Lambda},\Lambda,\omega)$ satisfies, for some $\delta>0$, for
any $p>0$ and $|\Lambda|$ sufficiently large,
\begin{equation}
  \label{eq:19}
  \pro\left(
    \left|\frac{N(J_{j,\Lambda},\Lambda,\omega)}{|N(J_{j,\Lambda})||\Lambda|}-1
    \right|\geq|\Lambda|^{-\delta}\right)\leq
  |\Lambda|^{-p}.  
\end{equation}
Let now $\mathcal{Z}^j_\Lambda$ be the set of configurations $\omega$
where one has both the conclusions of Theorem~\ref{thr:vbig1} and the
bound
\begin{equation}
  \label{eq:21}
  \left|\frac{N(J_{j,\Lambda},\Lambda,\omega)}{|N(J_{j,\Lambda})||\Lambda|}-1
  \right|\leq|\Lambda|^{-\delta}.
\end{equation}
By~\eqref{eq:11} and~\eqref{eq:19}, this new set still
satisfies~\eqref{eq:11}.\\
Define the following point measures:
\begin{itemize}
\item $\D \Xi^g_{J_{j,\Lambda}}(\omega,t,\Lambda):=
  \sum_{n\in\mathcal{N}^g_{\omega,j,\Lambda}}
  \delta_{|N(J_{j,\Lambda})||\Lambda|
    [N_{J_{j,\Lambda}}(E_n(\omega,\Lambda))-t]}$;
\item for $(\Lambda_\ell(\gamma_k))_k$, the cubes constructed in
  Theorem~\ref{thr:vbig1} (we write $\ell=\ell_\Lambda$), define the
  random variables:
  \begin{itemize}
  \item $X_{j,k}=X(\Lambda_\ell(\gamma_k),J_{j,\Lambda})$ is the
    Bernoulli random variable
    \begin{equation*}
      X_{j,k}=\car_{H_\omega(\Lambda_\ell(\gamma_k))\text{ has exactly one
          eigenvalue in }J_{j,\Lambda}\text{ with localization center in }
        \Lambda_{\ell-\ell'}}
    \end{equation*}
    where $\ell=\ell_\Lambda$ and $\ell'=\ell'_\Lambda$ are chosen
    as described above;
  \item $\tilde E_{j,k}=\tilde
    E(\Lambda_\ell(\gamma_k),J_{j,\Lambda})$ is this eigenvalue
    conditioned on the event $\{X_{j,k}=1\}$;
  \end{itemize}
  and the point measure $\D
  \Xi^{app}_{J_{j,\Lambda}}(\omega,t,\Lambda):=\sum_{k;\
    X_{j,k}=1}\delta_{|N(J_{j,\Lambda})|
    |\Lambda|[N_{J_{j,\Lambda}}(\tilde E_{j,k})-t]}$.
\end{itemize}
We consider these point measures as random processes under the uniform
distribution in $t$ in $[0,1]$.\\
We will need an estimate on the number
\begin{equation}
  \label{eq:28}
  N^{app}_{\omega,j,\Lambda}:=\{k;\ X_{j,k}=1\}. 
\end{equation}
It is provided by
\begin{Le}
  \label{le:5}
  For any $p>0$, for $|\Lambda|$ sufficiently large, one has
  \begin{equation*}
    \pro\left(\left|N^{app}_{\omega,j,\Lambda}-|N(J_{j,\Lambda})|
        |\Lambda|\right|
      \geq\left[|N(J_{j,\Lambda})||\Lambda|\right]^{2/3}\right)\leq
    e^{-\left[|N(J_{j,\Lambda})||\Lambda|\right]^{1/3}/3}\leq|\Lambda[|^{-p}.
  \end{equation*}
\end{Le}
\begin{proof}
  Lemma~\ref{le:5} follows by a standard large deviation argument for
  the i.i.d. Bernoulli random variables $(X_{j,k})_k$ as, by
  Lemma~\ref{le:7} and our choice of $J_{j,\Lambda}$ and
  $(\ell',\ell)$ (for $\nu\in(\xi,1)$ in Lemma~\ref{le:7}), their
  common distribution satisfies
  \begin{equation*}
    P(X_{j,k}=1)=|N(J_{j,\Lambda})||\Lambda_\ell|(1+o(1)).
  \end{equation*}
  The proof of Lemma~\ref{le:5} is complete.
\end{proof}
\noindent Thus, one may restrict once more the set of configurations
$\omega$ to those such that, for some $\delta>0$,
\begin{equation}
  \label{eq:24}
  \left|\frac{N^{app}_{\omega,j,\Lambda}}{|N(J_{j,\Lambda})||\Lambda|}-1\right|
  \leq|\Lambda|^{-\delta}.
\end{equation}
and call this set again $\mathcal{Z}^j_\Lambda$. By Lemma~\ref{le:5}
and~\eqref{eq:11}, the probability of this set also
satisfies~\eqref{eq:11} for any $p>0$ provided $|\Lambda|$ is
sufficiently large.\\
Using Lemma~\ref{le:4}, one then proves
\begin{Le}
  \label{le:3}
  For some $\beta>0$, for $\omega\in\mathcal{Z}^j_\Lambda$ and
  $\Lambda$ sufficiently large, one has,
  \begin{gather}
    \label{eq:14}
    \sup_{\varphi\in\mathcal{C}_{1,R}^+}\sup_{\substack{j\in
        G\\\omega\in\mathcal{Z}^j_\Lambda}}\left|\int_0^1
      e^{-\langle\Xi_{J_{j,\Lambda}}(\omega,t,\Lambda)
        ,\varphi\rangle}dt -\int_0^1e^{-\langle\Xi^g_{J_{j,\Lambda}}
        (\omega,t,\Lambda),\varphi\rangle}dt
    \right|\lesssim|\Lambda|^{-\beta},
    \\
    \intertext{and}\label{eq:17}
    \sup_{\varphi\in\mathcal{C}_{1,R}^+}\sup_{\substack{j\in
        G\\\omega\in\mathcal{Z}^j_\Lambda}}
    \left|\int_0^1e^{-\langle\Xi^g_{J_{j,\Lambda}}
        (\omega,t,\Lambda),\varphi\rangle}dt-
      \int_0^1e^{-\langle\Xi^{app}_{J_{j,\Lambda}}(\omega,t,\Lambda)
        ,\varphi\rangle}dt \right|\lesssim|\Lambda|^{-\beta}.
  \end{gather}
\end{Le}
\begin{proof}[The proof of Lemma~\ref{le:3}]
  As underlined above, the statements of Lemma~\ref{le:3} are
  corollaries of Lemma~\ref{le:4}.\\
  To obtain~\eqref{eq:14}, for $p=|\Lambda|$, it suffices to take
  \begin{itemize}
  \item $x_n^p=E_n(\omega,\Lambda)$ for $1\leq n\leq
    N(J_{j,\Lambda},\Lambda,\omega)$,
  \item $y_n^p=E_n(\omega,\Lambda)$ for
    $n\in\mathcal{N}^g_{\omega,j,\Lambda}$.
  \end{itemize}
  Assumption (2) in Lemma~\ref{le:4} is clearly fulfilled as
  $(x_n^p)_n$ is a subsequence of $(y_n^p)_n$. Assumption (1) is an
  immediate consequence~\eqref{eq:18} and~\eqref{eq:21}.\\
  Let us now prove~\eqref{eq:17}. Notice that, by
  Theorem~\ref{thr:vbig1}, one has $N^{app}_{\omega,j,\Lambda}\geq
  N^g_{\omega,j,\Lambda}$. Moreover, to each
  $n\in\mathcal{N}^g_{\omega,j,\Lambda}$, one can associate a unique
  $1\leq k(n)\leq N^{app}_{\omega,j,\Lambda}$ such that $X_{j,k(n)}=1$
  and the first part of~\eqref{eq:23} hold.\\
  To prove~\eqref{eq:17}, for $p=|\Lambda|$, it suffices to set
  \begin{itemize}
  \item $x_n^p=\tilde E_{j,k(n)}$ for $k(n)$ such that $X_{j,k(n)}=1$,
  \item $y_n^p=E_n(\omega,\Lambda)$ for
    $n\in\mathcal{N}^g_{\omega,j,\Lambda}$.
  \end{itemize}
  So we may take $K_p=N^g_{\omega,j,\Lambda}$. By the first part
  of~\eqref{eq:23}, we know that assumption (2) of Lemma~\ref{le:4} is
  satisfied with $\varepsilon_p=|\Lambda|^{-2}$. Thus,
  $\varepsilon_p\cdot K_p\lesssim|\Lambda|^{-1}$. \\
  That assumption (1) is satisfied follows immediately
  from~\eqref{eq:18} and~\eqref{eq:24}.\\
  This completes the proof of Lemma~\ref{le:3}
\end{proof}
\noindent So we have reduced the problem to analyzing the case of
i.i.d. random variables. In the next sections, we prove
\begin{Le}
  \label{le:9}
  Fix $\rho'\in(0,\rho/(1+\rho))$ where $\rho$ is defined by (M). Fix
  $\alpha\in(0,1)$ and $\nu\in(0,1)$ satisfying~\eqref{eq:87}
  and~\eqref{eq:40}. There exists $\kappa>1/d$ such that, for any
  $(\alpha_L)_{L\geq1}$ a sequence valued in $[1/2,2]$, one has
  \begin{equation*}
    \sum_{j\in G}\sum_{L\geq1}
    \esp\left(\left[\int_0^1
        e^{-\langle\Xi^{app}_{J_{j,\Lambda_{L^\kappa}}}
          (\omega,t,\Lambda_{L^\kappa}),\varphi_{\alpha_L}\rangle}dt-
        \exp\left(-\kappa_L\,\int_{-\infty}^{+\infty}
          \left(1-e^{-\varphi_{\alpha_L}(x)}\right)dx\right)
      \right]^2\right)<+\infty.
  \end{equation*}
\end{Le}
\noindent Let us now complete the proof of Lemma~\ref{le:2} using
Lemmas~\ref{le:3} and~\ref{le:9} and~\eqref{eq:11} the estimates on
the probability of $\mathcal{Z}^j_\Lambda$.\\
Clearly, Lemma~\ref{le:9} implies that
\begin{equation*}
  \esp\left(\limsup_{L\geq1}\sup_{j\in G}\left|\int_0^1
      e^{-\langle\Xi(\omega,t,j,\Lambda_{L^\kappa}),
        \varphi_{\alpha_L}\rangle}dt-
      \exp\left(-\kappa_L\,\int_{-\infty}^{+\infty}
        \left(1-e^{-\varphi_{\alpha_L}(x)}\right)dx\right)
    \right|\right)=0.
\end{equation*}
As all the integrands are bounded by $1$,
by~\eqref{eq:11},~\eqref{eq:14} and~\eqref{eq:17}, we know that
\begin{gather*}
  \esp\left(\limsup_{L\geq1}\sup_{j\in G}\left|\int_0^1
      e^{-\langle\Xi_{J_{j,\Lambda_{L^\kappa}}}(\omega,t,
        \Lambda_{L^\kappa}),\varphi_{\alpha_L}\rangle}dt
      -\int_0^1e^{-\langle\Xi^g_{J_{j,\Lambda_{L^\kappa}}}
        (\omega,t,\Lambda_{L^\kappa}),\varphi_{\alpha_L}\rangle}dt
    \right|\right)=0\\ \intertext{ and }
  \esp\left(\limsup_{L\geq1}\sup_{j\in G}
    \left|\int_0^1e^{-\langle\Xi^g_{J_{j,\Lambda_{L^\kappa}}}
        (\omega,t,\Lambda_{L^\kappa}),\varphi_{\alpha_L}\rangle}dt-
      \int_0^1e^{-\langle\Xi^{app}_{J_{j,\Lambda_{L^\kappa}}}
        (\omega,t,\Lambda_{L^\kappa}) ,\varphi_{\alpha_L}\rangle}dt
    \right|\right)=0.
\end{gather*}
Thus, if $\alpha_L\to1$ when $L\to+\infty$, one has
\begin{equation*}
  \exp\left(-\kappa_L\,\int_{-\infty}^{+\infty}
    \left(1-e^{-\varphi_{\alpha_L}(x)}\right)dx\right)
  \vers_{L\to+\infty}\exp\left(-\int_{-\infty}^{+\infty}
    \left(1-e^{-\varphi(x)}\right)dx\right),
\end{equation*}
this clearly implies~\eqref{eq:71} and completes the proof of
Lemma~\ref{le:2}.
\subsection{The proof of Lemma~\ref{le:9}}
\label{sec:proof-lemma-2}
Let us recall a few facts that will be of use in this proof. \\
Write $\Lambda_\ell=\Lambda_\ell(0)$ and define the random variables
$X$ and $\tilde E$ as in the beginning of
section~\ref{sec:distr-eigenv-small} for $I_\Lambda=J_{j,\Lambda}$ and
the cube $\Lambda_\ell$. Recall that the cube $\Lambda=\Lambda_L$ is
much larger than $\Lambda_\ell$. Now, pick
$N^{app}_{\omega,j,\Lambda}$ independent copies of $\tilde E$, say
$(\tilde E_k)_{1\leq k\leq N^{app}_{\omega,j,\Lambda}}$. Then, the
random process $\Xi^{app}_{J_{j,\Lambda}}$ is the process
\begin{equation*}
  \Xi^{app}_{J_{j,\Lambda}}(\omega,t,\Lambda):=
  \sum_{1\leq k\leq N^{app}_{\omega,j,\Lambda}}\delta_{|N(J_{j,\Lambda})|
    |\Lambda|[N_{J_{j,\Lambda}}(\tilde E_k)-t]}.
\end{equation*}
By Lemma~\ref{le:4} and~\eqref{eq:24}, it thus suffices to study the
point process
\begin{equation}
  \label{eq:30}
  \Xi(\omega,t,j,\Lambda):=
  \sum_{1\leq k\leq |\Lambda||N(J_{j,\Lambda})|}\delta_{|N(J_{j,\Lambda})|
    |\Lambda|[N_{J_{j,\Lambda}}(\tilde E_k)-t]}.
\end{equation}
Recall that $N_{J_{j,\Lambda}}$ is defined by~\eqref{eq:7} for
$J=J_{j,\Lambda}$. Pick $\varphi\in\mathcal{C}_{1,R}^+$
(see~\eqref{eq:65}). As the random variables $(\tilde E_k)_{1\leq
  k\leq |N(J_{j,\Lambda})||\Lambda|}$ are i.i.d., one computes
\begin{equation}
  \label{eq:32}
  \esp\left(\int_0^1e^{-\langle\Xi(\omega,t,j,\Lambda),\varphi\rangle}dt\right)
  =\int_0^1\Phi(t,\Lambda,J_{j,\Lambda},\varphi)dt
\end{equation}
and
\begin{equation}
  \label{eq:33}
  \esp\left(\left[\int_0^1e^{-\langle\Xi(\omega,t,j,\Lambda)
        ,\varphi\rangle}dt\right]^2\right) =\int_0^1\int_0^1
  \Phi(t,t',\Lambda,J_{j,\Lambda},\varphi)dtdt'
\end{equation}
where
\begin{gather}
  \label{eq:34}
  \Phi(t,\Lambda,J_{j,\Lambda},\varphi)=
  \left[1-\esp\left(1-e^{-\varphi(|N(J_{j,\Lambda})|
        |\Lambda|[N_{J_{j,\Lambda}}(\tilde E)-t])}\right)
  \right]^{|N(J_{j,\Lambda})||\Lambda|}\\\intertext{and}
  \label{eq:35}
  \begin{split}
    &\Phi(t,t',\Lambda,J_{j,\Lambda},\varphi)\\
    &=\left[1-\esp\left(1-e^{-\varphi(|N(J_{j,\Lambda})|
          |\Lambda|[N_{J_{j,\Lambda}}(\tilde
          E)-t])-\varphi(|N(J_{j,\Lambda})|
          |\Lambda|[N_{J_{j,\Lambda}}(\tilde E)-t'])}\right)
    \right]^{|N(J_{j,\Lambda})||\Lambda|}.
  \end{split}
\end{gather}
If $E\mapsto N_{J_{j,\Lambda}}(E)$ were the distribution function of
the random variable $\tilde E$, the random variables
$N_{J_{j,\Lambda}}(\tilde E)$ would be distributed uniformly on
$[0,1]$ and the desired result would be standard and follow e.g. from
the computations done in the appendix of~\cite{Mi:11}. The
distribution function of $\tilde E$ is described by
Lemma~\ref{le:7}. As we only consider $j\in G$, we know that
$|N(J_{j,\Lambda})|\geq|J_{j,\Lambda}|^{1+\rho'}$ for some
$\rho'\in(0,\rho/(1+(d+1)\rho))$. Thus, choosing $\nu\in(\xi,1)$ in
Lemma~\ref{le:7}, for $x\in J_{j,\Lambda}$ (take $y=0$),
using~\eqref{eq:51} and~\eqref{eq:87}, the estimation~\eqref{eq:52}
becomes, for any $p>0$ and $|\Lambda|=|\Lambda_L|$ sufficiently large,
\begin{equation}
  \label{eq:31}
  \begin{split}
    \left|\kappa_{\Lambda}\cdot|N(J_{j,\Lambda})||\Lambda|\,\tilde\Xi(x)
      -|N(J_{j,\Lambda})||\Lambda|\,N_{J_{j,\Lambda}}(x)\right|&
    \lesssim|N(J_{j,\Lambda})||\Lambda|[|J_{j,\Lambda}||\Lambda_\ell|]^{\rho}\\
    &\lesssim|N(J_{j,\Lambda})|^{1-\alpha^{-1}-d\nu\rho+\rho/(1+\tilde\rho)}\\
    &\lesssim|N(J_{j,\Lambda})|^{\rho'/(1+\rho')}
  \end{split}
\end{equation}
where, by~\eqref{eq:51} and the same computation as in~\eqref{eq:31},
one has, for some $\beta>0$,
\begin{equation}
  \label{eq:42}
  \kappa_{\Lambda}:=\frac{\pro(X(\Lambda_{\ell_\Lambda},J_{j,\Lambda},
    \ell'_\Lambda)=1)} {|N(J_{j,\Lambda})||\Lambda_\ell|}
  =1+O(|N(J_{j,\Lambda})|^{(\rho-\rho')/(1+\rho')-d\nu\rho}). 
\end{equation}
Using~\eqref{eq:31}, from~\eqref{eq:34}, as
$\varphi\in\mathcal{C}_{1,R}^+$, we derive
\begin{multline}
  \label{eq:20}
  \frac{\log\Phi(t,\Lambda,J_{j,\Lambda},\varphi)}
  {|N(J_{j,\Lambda})||\Lambda|}\\=\log
  \left[1-\esp\left(1-e^{-\varphi(|N(J_{j,\Lambda})|
        |\Lambda|[\kappa_{\Lambda}\cdot\tilde\Xi(\tilde E)-t])}\right)
    +O\left(|N(J_{j,\Lambda})|^{\rho'/(1+\rho')}\right)\right].
\end{multline}
Now, fix $\kappa\in(0,1)$. The random variable $\tilde\Xi(\tilde E)$
is uniformly distributed on $[0,1]$; thus, we compute
\begin{equation}
  \label{eq:39}
  \begin{split}
    \esp\left(1-e^{-\varphi(|N(J_{j,\Lambda})|
        |\Lambda|[\kappa_{\Lambda}\cdot\tilde\Xi(\tilde
        E)-t])}\right)&=
    \int_0^1\left(1-e^{-\varphi(|N(J_{j,\Lambda})|
        |\Lambda|[\kappa_{\Lambda}\,u-t])}\right)du\\&\hskip-2cm=
    \frac1{\kappa_{\Lambda}|N(J_{j,\Lambda})|
      |\Lambda|}\int_{-|N(J_{j,\Lambda})|
      |\Lambda|t}^{|N(J_{j,\Lambda})|
      |\Lambda|[\kappa_{\Lambda}-t]}\left(1-e^{-\varphi(u)}\right)du
    \\&\hskip-2cm=\frac1{\kappa_{\Lambda}|N(J_{j,\Lambda})|
      |\Lambda|}\int_{-\infty}^{+\infty}
    \left(1-e^{-\varphi(u)}\right)du
\end{split}
\end{equation}
if, using~\eqref{eq:58}, we assume that $t$ satisfies
\begin{equation}
  \label{eq:38}
  |N(J_{j,\Lambda})|^{(\alpha^{-1}-1)\kappa}\leq t\leq 1-
  |N(J_{j,\Lambda})|^{(\alpha^{-1}-1)\kappa}
  +O(|N(J_{j,\Lambda})|^{(\rho-\rho')/(1+\rho')-d\nu\rho})
\end{equation}
for $|\Lambda|$ sufficiently large (as $\alpha\in(0,1)$ and
$|N(J_{j,\Lambda})|\to0$ when $|\Lambda|\to+\infty$). Here, we have
used~\eqref{eq:42}.\\
Now, if we take $\alpha$ in~\eqref{eq:58} so small that~\eqref{eq:40}
be satisfied then,~\eqref{eq:20} and~\eqref{eq:39} yield that, for any
$\beta\in(0,\rho'/(1+\rho'))$, for $t$ satisfying~\eqref{eq:38} and
$|\Lambda|$ sufficiently large,
\begin{equation*}
  \log\Phi(t,\Lambda,J_{j,\Lambda},\varphi)=
  \frac1{\kappa_{\Lambda}}\int_{-\infty}^{+\infty}
  \left(1-e^{-\varphi(u)}\right)du+O(|N(J_{j,\Lambda})|^{\beta}).
\end{equation*}
Thus, by~\eqref{eq:42}, for
$\beta\in(0,\min(\rho',\rho-\rho')/(1+\rho'))$, for $t$
satisfying~\eqref{eq:38} and $|\Lambda|$ sufficiently large,
\begin{equation*}
  \log\Phi(t,\Lambda,J_{j,\Lambda},\varphi)=
  \int_{-\infty}^{+\infty}
  \left(1-e^{-\varphi(u)}\right)du+O(|N(J_{j,\Lambda})|^{\beta}).
\end{equation*}
Using~\eqref{eq:31}, from~\eqref{eq:35}, as
$\varphi\in\mathcal{C}_{1,R}^+$, we derive
\begin{equation}
  \label{eq:15}
  \begin{split}
    &\frac{\log\Phi(t,t',\Lambda,J_{j,\Lambda},\varphi)}
    {|N(J_{j,\Lambda})||\Lambda|}\\
    &\hskip1cm=\log\left[1-\esp\left(1-e^{-\varphi(|N(J_{j,\Lambda})|
          |[\kappa_{\Lambda}\cdot\tilde\Xi(\tilde
          E)-t])-\varphi(|N(J_{j,\Lambda})|
          |\Lambda|[\kappa_{\Lambda}\tilde\Xi(\tilde E)-t'])}\right)\right.\\
      &\hskip9cm\left.+O(|N(J_{j,\Lambda})|^{\rho'/(1+\rho')})
    \right].
  \end{split}
\end{equation}
Moreover, for $t$ and $'t$ satisfying~\eqref{eq:38} such that
\begin{equation}
  \label{eq:37}
  |N(J_{j,\Lambda})|^{(\alpha^{-1}-1)\kappa}\leq |t-t'|
\end{equation}
as above, one computes
\begin{equation}
  \label{eq:43}
  \begin{split}
    &\esp\left(1-e^{-\varphi(|N(J_{j,\Lambda})|
        |[\kappa_{\Lambda}\cdot\tilde\Xi(\tilde
        E)-t])-\varphi(|N(J_{j,\Lambda})|
        |\Lambda|[\kappa_{\Lambda}\tilde\Xi(\tilde E)-t'])}\right)\\&\hskip2cm=
    \int_0^1\left(1-e^{-\varphi(|N(J_{j,\Lambda})|
        |\Lambda|[\kappa_{\Lambda}\,u-t])-\varphi(|N(J_{j,\Lambda})|
        |\Lambda|[\kappa_{\Lambda}\,u-t'])}\right)du\\
    &\hskip2cm=\frac2{\kappa_{\Lambda}|N(J_{j,\Lambda})|
      |\Lambda|}\int_{-\infty}^{+\infty}
    \left(1-e^{-\varphi(u)}\right)du
\end{split}
\end{equation}
for $|\Lambda|$ sufficiently large. Here, we have
used~\eqref{eq:42}.\\
Again, if we take $\alpha$ in~\eqref{eq:58} so small
that~\eqref{eq:40} is satisfied then,~\eqref{eq:15},~\eqref{eq:42}
and~\eqref{eq:43} yield that, for any
$\beta\in(0,\min(\rho',\rho-\rho')/(1+\rho'))$, for $t$ and $'t$
satisfying~\eqref{eq:38} and~\eqref{eq:37}, for $|\Lambda|$
sufficiently large,
\begin{equation*}
  \log\Phi(t,t',\Lambda,J_{j,\Lambda},\varphi)=
  2\int_{-\infty}^{+\infty}
  \left(1-e^{-\varphi(u)}\right)du+O(|N(J_{j,\Lambda})|^{\beta}).
\end{equation*}
Finally notice that $\Phi(t,\Lambda,J_{j,\Lambda},\varphi)$ and
$\Phi(t,t',\Lambda,J_{j,\Lambda},\varphi)$ are both bounded by $1$ and
that the measure of the sets of $t\in[0,1]$ satisfying~\eqref{eq:38}
and the measure of the sets of $(t,t')\in[0,1]^2$
satisfying~\eqref{eq:38} for $t$ and $t'$ and~\eqref{eq:37} are both
larger than
\begin{equation*}
  1-O(|N(J_{j,\Lambda})|^{(\alpha^{-1}-1)\kappa})
  +O(|N(J_{j,\Lambda})|^{(\rho-\rho')/(1+\rho')-d\nu\rho}).
\end{equation*}
Thus, thus taking into account~\eqref{eq:58}, we have proved
\begin{Le}
  \label{le:10}
  Fix $R>0$. Fix $\rho'\in(0,\rho/(1+\rho))$ where $\rho$ is defined
  by (M). Fix $\alpha\in(0,1)$ and $\nu\in(0,1)$
  satisfying~\eqref{eq:87} and~\eqref{eq:40}.\\
  There exists $\beta>0$ such that, for $|\Lambda|$ sufficiently large
  (depending only on $R$, $\rho'$, $\alpha$ and $\nu$), one has
  \begin{equation}
    \label{eq:36}
    \sup_{\varphi\in\mathcal{C}_{1,R}^+} \sup_{j\in
      G}\left|\int_0^1\Phi(t,\Lambda,J_{j,\Lambda},\varphi)dt-
      \exp\left(-\int_{-\infty}^{+\infty}
        \left(1-e^{-\varphi(x)}\right)dx\right)\right|\leq
    |\Lambda|^{-\beta}
  \end{equation}
  and
  \begin{multline}
    \label{eq:68}
    \sup_{\varphi\in\mathcal{C}_{1,R}^+} \sup_{j\in
      G}\left|\int_0^1\int_0^1\Phi(t,t',\Lambda,J_{j,\Lambda},\varphi)
      dtdt'-\exp\left(-2\int_{-\infty}^{+\infty}
        \left(1-e^{-\varphi(x)}\right)dx\right)\right| \\\leq
    |\Lambda|^{-\beta}.
  \end{multline}
\end{Le}
\noindent Let us use Lemma~\ref{le:10} to complete the proof of
Lemma~\ref{le:9}. For $L\geq1$, let $\Lambda=\Lambda_L$.  Fix
$(\alpha_L)_{L\geq1}$ a sequence valued in $[1/2,2]$. Then, for
$\varphi\in\mathcal{C}_{1,R}^+$, the sequence
$(\varphi_{\alpha_L})_{L\geq1}$ is bounded in
$\mathcal{C}_{1,2R}^+$. Thus, by Lemma~\ref{le:10}, for $\kappa$ such
that $\kappa\beta d>1$ and $(\alpha_L)_{L\geq1}$, any sequence valued
in $[1/2,2]$, we have that
\begin{equation*}
  \sum_{j\in G}\sum_{L\geq1}
  \esp\left(\left[\int_0^1
      e^{-\langle\Xi(\omega,t,j,\Lambda_{L^\kappa}),\varphi_{\alpha_L}\rangle}dt-
      \exp\left(-\int_{-\infty}^{+\infty}
        \left(1-e^{-\varphi_{\alpha_L}(x)}\right)dx\right)
    \right]^2\right)<+\infty.
\end{equation*}
Thus, if $\alpha_L\to1$ as $L\to+\infty$, we have proved
Lemma~\ref{le:9}.\qed
\subsection{The proof of Lemma~\ref{le:8}}
\label{sec:proof-lemma-1}
Clearly, by~\eqref{eq:74} and~\eqref{eq:73}, to prove
Lemma~\ref{le:8}, it suffices to show that, for some $\beta>0$,
$\omega$-almost surely, one has
\begin{equation}
  \label{eq:70}
  \sup_{\substack{j\in G\\L^\kappa\leq L'\leq (L+1)^\kappa}}
  \left|\int_0^1
    e^{-\langle\Xi_{J_{j,\Lambda_{L^\kappa}}}(\omega,t,\Lambda_{L'}),
      \varphi\rangle}dt-\int_0^1
    e^{-\langle\Xi_{J_{j,\Lambda_{L^\kappa}}}(\omega,t,\Lambda_{L^\kappa}),
      \varphi_{\alpha_{L'}}\rangle}dt\right|\lesssim L^{-\beta}
\end{equation}  
where $\alpha_{L'}=|\Lambda_{L'}|/|\Lambda_{L^\kappa}|$. Notice here
that we chose the same partition of $J$ into
$(J_{j,\Lambda_{L^\kappa}})_j$ for all $L^\kappa\leq L'\leq
(L+1)^\kappa$ which is possible as
$|\Lambda_{L'}|=|\Lambda_{L^{\nu}}|(1+o(1))$.\\
For $\Lambda'\subset\Lambda$, let $E_1(\omega,\Lambda,\Lambda')\leq
E_2(\omega,\Lambda,\Lambda')\leq \cdots\leq
E_{N(J,\Lambda,\Lambda',\omega)}(\omega,\Lambda,\Lambda')$ be the
eigenvalues of $H_\omega(\Lambda)$ in $J$ with localization center in
$\Lambda'$, and, thus, $N(J,\Lambda,\Lambda',\omega)$ be their number
which is random. Recall that
$N(J,\Lambda,\omega)=N(J,\Lambda,\Lambda,\omega)$ denotes the number
of eigenvalues of  $H_\omega(\Lambda)$ in $J$.\\
In Lemma~\ref{le:6}, we prove that most eigenvalues of
$H_\omega(\Lambda_{L'})$ and of $H_\omega(\Lambda_{L^{\nu}})$ in $J$
have center of localization in $\Lambda_{(L-1)^\nu}$; this is
essentially a consequence of the description given by
Theorem~\ref{thr:vbig1}. Thus, by Lemma~\ref{le:14}, these eigenvalues
of $H_\omega(\Lambda_{L'})$ and of $H_\omega(\Lambda_{L^{\nu}})$ are
close to one another. We can then use Lemma~\ref{le:4} to compare
$\Xi_{J_{j,\Lambda_{L^\kappa}}}(\omega,t,\Lambda_{L'})$ and
$\Xi_{J_{j,\Lambda_{L^\kappa}}}(\omega,t,\Lambda_{L^\kappa})$.\\
We prove
\begin{Le}
  \label{le:6}
  Pick $\nu>0$. There exists $\beta>0$ such that, $\omega$-almost
  surely, for $L$ sufficiently large and $L^\kappa\leq L'\leq
  (L+1)^\kappa$ and $j\in G$, one has
  \begin{enumerate}
  \item
    \begin{equation*}
      \left|\frac{N(J_{j,\Lambda_{L^\kappa}},\Lambda_{L'},
          \Lambda_{(L-1)^\nu},\omega)}
        {N(J_{j,\Lambda_{L^\kappa}},\Lambda_{L^\kappa},
          \Lambda_{(L-1)^\nu},\omega)}-1\right|+
      \left|\frac{N(J_{j,\Lambda_{L^\kappa}},\Lambda_{L'},
          \Lambda_{(L-1)^\nu},\omega)}
        {N(J_{j,\Lambda_{L^\kappa}},\Lambda_{L'},\omega)}-1\right|
      \lesssim L^{-\beta};
    \end{equation*}
  \item
    \begin{equation*}
      \left|\frac{N(K_{j,\Lambda_{L^\kappa}},\Lambda_{L'},\omega)}
        {N(J_{j,\Lambda_{L^\kappa}},\Lambda_{L'},\omega)}-1\right|+
      \left|\frac{N(K_{j,\Lambda_{L^\kappa}},\Lambda_{L^\kappa},\omega)}
        {N(J_{j,\Lambda_{L^\kappa}},\Lambda_{L^\kappa},\omega)}-1\right|
      \lesssim L^{-\beta}
    \end{equation*}
    where $(K_{j,\Lambda})_j$ are defined in the beginning of
    section~\ref{sec:reduct-local-behav};
  \item to each eigenvalue of $H_\omega(\Lambda_{L'})$ in
    $K_{j,\Lambda_{L^\kappa}}$ with localization center in
    $\Lambda_{(L-1)^\nu}$, say, $E$, one can associate an eigenvalue
    of $H_\omega(\Lambda_{L^\kappa})$ in $J_{j,\Lambda_{L^\kappa}}$, say,
    $E'$, such that $|E-E'|\leq L^{-3d\nu}$;
  \item to each eigenvalue of $H_\omega(\Lambda_{L^\kappa})$ in $J$ with
    localization center in $\Lambda_{(L-1)^\nu}$ in
    $K_{j,\Lambda_{L^\kappa}}$, say, $E$, one can associate an eigenvalue
    of $H_\omega(\Lambda_{L'})$ in $J_{j,\Lambda_{L^\kappa}}$, say, $E'$,
    such that $|E-E'|\leq L^{-3d\nu}$.
  \end{enumerate}
\end{Le}
\noindent We now can apply Lemma~\ref{le:9} to
$\langle\Xi_{J_{j,\Lambda_{L^\kappa}}}(\omega,t,\Lambda_{L'}),\varphi\rangle$
and $\langle\Xi_{J_{j,\Lambda_{L^\kappa}}}(\omega,t,\Lambda_{L^\kappa}),
\varphi_{\alpha_L'}\rangle$. By Lemma~\ref{le:6}, the assumptions of
Lemma~\ref{le:9} will be satisfied if, using the notations of
Lemma~\ref{le:9}, we take
\begin{itemize}
\item $X_p$ to be the eigenvalues of $H_\omega(\Lambda_{L'})$ in
  $K_{j,\Lambda_{L^\kappa}}$ with localization center in
  $\Lambda_{(L-1)^\nu}$,
\item $Y_p$ to be the eigenvalues of $H_\omega(\Lambda_{L^\kappa})$ in
  $K_{j,\Lambda_{L^\kappa}}$ with localization center in
  $\Lambda_{(L-1)^\nu}$.
\end{itemize}
Indeed, Lemma~\ref{le:6} then provides the estimates
\begin{equation*}
  0\leq a_p \lesssim L^{-\beta},\quad 0\leq K_p\leq
  CL^{d\nu+1}\quad\text{ and }\quad0\leq \varepsilon_p\leq L^{-3d\nu}. 
\end{equation*}
Then,~\eqref{eq:75} and, thus, Lemma~\ref{le:8}, is an immediate
consequence of Lemma~\ref{le:9} (where one of the functions $\varphi$
has been replaced with $\varphi_{\alpha_{L'}}$). This completes the
proof of Lemma~\ref{le:8}.\qed
\begin{proof}[Proof of Lemma~\ref{le:6}]
  First, in Theorem~\ref{thr:vbig1} (see the proofs in~\cite{Ge-Kl:10}
  for more details), for $L^\kappa\leq L'\leq (L+1)^\nu$, one can pick
  the same scale $\ell$. Then, by Theorem~\ref{thr:vbig1} (for
  $R>2\nu$), for any $p>0$ and some $\beta>0$, we know that, with a
  probability at least $1-L^{-p}$, for $L^\kappa\leq L'\leq (L+1)^\nu$
  and $j\in G$ (recall that $\#G\leq|\Lambda_L|^{-\beta}$), up to at
  most $N(J_{j,\Lambda_{L^\kappa}})|\Lambda_{L'}||\Lambda_L|^{-\beta}$ of
  them, the eigenvalues of $H_\omega(\Lambda_{L'})$ in
  $J_{j,\Lambda_{L^\kappa}}$ are given by those of the operators
  $(H_\omega(\Lambda_{\ell}(\gamma))_\gamma)$ up to an error bounded
  by $|\Lambda|^{-2}$. In particular, up to at most
  $N(J_{j,\Lambda_{L^\kappa}})|\Lambda_{L'}||\Lambda_L|^{-\beta}$ of
  them, the eigenvalues of $H_\omega(\Lambda_{L'})$ in
  $J_{j,\Lambda_{L^\kappa}}$ with localization center in
  $\Lambda_{(L-1)^\nu}$ and of $H_\omega(\Lambda_{L^\kappa})$ in
  $J_{j,\Lambda_{L^\kappa}}$ with localization center in
  $\Lambda_{(L-1)^\nu}$ are the same up to an error bounded by
  $CL^{-2d}$. Moreover, the number of cubes
  $(\Lambda_{\ell}(\gamma))_\gamma$ that are not contained in
  $\Lambda_{(L-1)^\nu}$ is bounded by $CL^{\nu d-1}$ which is itself
  bounded by $CN(J_{j,\Lambda_{L^\kappa}})|\Lambda_{L'}|
  |\Lambda_L|^{-\beta}$. Thus, if one pick $p>1$, the Borel-Cantelli
  Lemma tells us that (1), (3) and (4) of Lemma~\ref{le:6} are almost
  surely  fulfilled.\\
  To prove that (2) is also almost surely true, we use the estimates
  on large deviations given by Theorem~\ref{thr:4} on the sets
  $J_{j,\Lambda_{L^\kappa}}\setminus K_{j,\Lambda_{L^\kappa}}$ that are of
  size $L^{-d/2}$. We thus obtain that, with probability at least
  $1-e^{-L^{d\nu/4}}$, for $L^\kappa\leq L'\leq (L+1)^\nu$, the number of
  eigenvalues of of $H_\omega(\Lambda_{L'})$ in
  $J_{j,\Lambda_{L^\kappa}}\setminus K_{j,\Lambda_{L^\kappa}}$ is bounded by
  $N(J_{j,\Lambda_{L^\kappa}})|\Lambda_{L'}|L^{-1}$. Thus, using again
  the Borel-Cantelli Lemma and (1), we obtain (2).\\
  This completes the proof of Lemma~\ref{le:6}.
\end{proof}
\subsection{The proof of Theorem~\ref{thr:3}}
\label{sec:proof-theorem-1}
It follows the same analysis as the proof of Theorem~\ref{thr:2};
thus, we do not give any details. We distinguish two cases. First if
$|N(I_\Lambda)|\leq |\Lambda|^{-\alpha}$ for $\alpha$ chosen as
prescribed in section~\ref{sec:reduct-local-behav} (see also the
comments following Theorem~\ref{thr:vbig1}). In this case, we can
apply Theorem~\ref{thr:vbig1} to the interval $I_\Lambda$ as it
satisfies all the assumptions of Theorem~\ref{thr:vbig1} if we choose
the scales $\ell_\Lambda\asymp |N(I_\Lambda)|^{-\nu}$ for some $\nu$
satisfying~\eqref{eq:25}. We then follow the proof of
Theorem~\ref{thr:2} for this single interval to obtain
Theorem~\ref{thr:3}. If $|N(I_\Lambda)|\geq |\Lambda|^{-\alpha}$, we
again split the interval into intervals of size $|\Lambda|^{-\alpha}$
to apply Theorem~\ref{thr:vbig1} to each of those, actually, to most
of those. Indeed, up to the renormalization of $N$ so that it has unit
mass on $I_\Lambda$ we have brought ourselves back to the proof of
Theorem~\ref{thr:2}
\begin{Rem}
  \label{rem:4}
  We see that the first condition in~\eqref{eq:63} is needed only when
  the interval $I_\Lambda$ is very small. Actually, one needs it for
  $|I_\Lambda|$ smaller than $|\Lambda|^{-\nu}$ for some $\nu>0$.
\end{Rem}
\noindent The condition~\eqref{eq:64} is needed to obtain the results
corresponding to Lemmas~\ref{le:8} and~\ref{le:6}. In
Lemmas~\ref{le:8} and~\ref{le:6}, the error estimate is of size an
inverse power of $L$; in the corresponding result in the present
setting, it is replaced by $o(1)$ coming from condition~\eqref{eq:64}.
\subsection{The proof of Theorem~\ref{thr:5}}
\label{sec:proof-theorem-2}
Theorem~\ref{thr:5} follows from Theorem~\ref{thr:2}, Lemma~\ref{le:4}
and the fact that most eigenvalues of $H_\omega$ in $J$ with
localization center in $\Lambda$ are very well approximated by an
eigenvalue of $H_\omega(\Lambda)$ in $J$, and vice versa.\\
Write $J=[a,b]$. Using the techniques of the proof of
Lemma~\ref{le:6}, one proves the following result for the eigenvalues
of $H_\omega$ is $J$ having localization center in $\Lambda$
\begin{Le}
  \label{le:12}
  Fix $\nu\in(0,1)$. There exists $\beta>0$ such that, $\omega$-almost
  surely, for $L$ sufficiently large, one has
  \begin{enumerate}
  \item
    \begin{equation*}
      \left|\frac{N^f(J,\Lambda,\omega)}{N(J,\Lambda,\omega)}-1\right|
      \leq|\Lambda|^{-\beta};
    \end{equation*}
  \item to each eigenvalue of $H_\omega(\Lambda_L)$ in
    $J_L:=[a+L^{-3d/2},b-L^{-3d/2}]$ with localization center in
    $\Lambda_{L-L^\kappa}$, say, $E$, one can associate an eigenvalue
    of $H_\omega$ in $J$ with localization center in $\Lambda_L$, say,
    $E'$, such that $|E-E'|\leq L^{-2d}$;
  \item to each eigenvalue of $H_\omega$ in $J_L$ with localization
    center in $\Lambda_{L-L^\kappa}$, say, $E$, one can associate an
    eigenvalue of $H_\omega(\Lambda_L)$ in $J$, say, $E'$, that
    satisfies $|E-E'|\leq L^{-2d}$.
  \end{enumerate}
\end{Le}
\noindent One then uses this to combine Theorem~\ref{thr:2} and
Lemma~\ref{le:4} to obtain Theorem~\ref{thr:5}.
\section{The proof of Theorems~\ref{thr:8} and~\ref{thr:7}}
\label{sec:proof-theorems}
These proofs are simple and rely on general theorems on
transformations of point processes (see
e.g.~\cite[Chap. 5.5]{MR1700749} and \cite[Chap. 3.5]{MR2364939}).
\subsection{The proof of Theorem~\ref{thr:8}}
\label{sec:proof-theorem-3}
As in the proof of Theorem~\ref{thr:2}, it suffices to consider the
case when $J$ is an interval in the essential support of $\nu$, that
is, $N$ is strictly increasing on $J$. In particular, one has
$\nu(t)>0$ for almost every $t\in J$. \\
If $t$ is a random variable distributed according to the law
$\nu_J(t)dt$, then $\tilde t:=N_J(t)$ is uniformly distributed on
$[0,1]$. Thus, the process $\Xi_J(\omega,\tilde t,\Lambda)$ under the
uniform law in $\tilde t$ has the same law as the process
$\Xi_J(\omega,N_J(t),\Lambda)$ under the law $\nu_J(t)dt$. \\
Rewrite the point measures $\Xi_J(\omega,N_J(t),\Lambda)$ and
$\tilde\Xi_J(\omega,t,\Lambda)$ as
\begin{equation*}
  \Xi_J(\omega,N_J(t),\Lambda) = \sum_{E_n(\omega,\Lambda)\in J}
  \delta_{x_n(\omega,t)}\quad\text{ and }\quad
  \tilde\Xi_J(\omega,t,\Lambda) = \sum_{E_n(\omega,\Lambda)\in J}
  \delta_{\tilde x_n(\omega,t)}
\end{equation*}
where
\begin{gather*}
  x_n(\omega,t):=|N(J)||\Lambda|[N_J(E_n(\omega,\Lambda))-N_J(t)]
  =|\Lambda|[N(E_n(\omega,\Lambda))-N(t)]\\
  \intertext{ and } \tilde
  x_n(\omega,t):=\nu(t)|\Lambda|[E_n(\omega,\Lambda)-t].
\end{gather*}
Thus, one has
\begin{equation}
  \label{eq:78}
  x_n(\omega,t)=\varpi_\Lambda(\tilde
  x_n(\omega,t);t)\quad\text{ and }\quad \tilde
  x_n(\omega,t)=\chi_\Lambda(x_n(\omega,t);t)
\end{equation}
where
\begin{gather*}
  \varpi_\Lambda(x;t)=|\Lambda|\left[N\left(t+\frac{x}{\nu(t)|\Lambda|}\right)
    - N(t)\right]\\\intertext{and}\chi_\Lambda(x;t)=\nu(t)|\Lambda|
  \left[N^{-1}\left(N(t)+\frac{x}{|\Lambda|}\right)-t\right]
\end{gather*}
where $N^{-1}$ is the inverse of the Lipschitz continuous, strictly
increasing function $N$.\\
Note that, if $N(J,\Lambda,\omega)$ denotes the number of eigenvalues
of $H_\omega(\Lambda)$ in $J$, one has
\begin{equation}
  \label{eq:77}
  t=\frac1{N(J,\Lambda,\omega)}\cdot
  N^{-1}\left(\sum_{E_n(\omega,\Lambda)\in J}N(E_n(\omega))
    -\frac{x_n}{|\Lambda|}\right).
\end{equation}
Following the notations of~\cite{MR2364939}, let $\mathcal{M}_p(\R)$
denote the space of point measures on the real line endowed with its
standard metric structure. Actually, by Minami's estimate (M), we
could restrict ourselves to working with simple point measures.\\
The point processes $\Xi_J(\omega,N_J(t),\Lambda)$ and
$\tilde\Xi_J(\omega,t,\Lambda)$ under the law $\nu_J(t)dt$ are the
random processes (i.e. the Borelian random variables) obtained as
push-forwards of the probability measure $\nu_J(t)dt$ through the maps
$\D t\in\R\mapsto\Xi_J(\omega,N_J(t),\Lambda)\in\mathcal{M}_p(\R)$ and
$\D t\in\R\mapsto\tilde\Xi_J(\omega,t,\Lambda)\in\mathcal{M}_p(\R)$.
We denote them respectively by $\Xi_J(\omega,\Lambda)$ and
$\tilde\Xi_J(\omega,\Lambda)$.\\
One can extend the mapping $x\in\R\mapsto\chi_\Lambda(x,t)\in\R$ to a
map, say, $\chi_{\omega,\Lambda}$ on point measures in
$\mathcal{M}_p(\R)$ on the real line by just mapping the supports
pointwise onto one another and computing $t$ using~\eqref{eq:77} i.e.
\begin{equation*}
  \chi_{\omega,\Lambda}\left(\sum_{n}a_n\delta_{x_n}\right)
  =\sum_{n}a_n\delta_{\chi_{\omega,\Lambda}
    \left(x_n;t\left(\sum_{n}a_n\delta_{x_n}\right)\right)}
\end{equation*}
where $t\left(\sum_{n}a_n\delta_{x_n}\right)$ is defined as
\begin{equation*}
  t\left(\sum_{n}a_n\delta_{x_n}\right)=
  \frac1{N(J,\Lambda,\omega)}\sum_{E_n(\omega,\Lambda)\in
    J}N^{-1}\left(N(E_n(\omega))-\frac{x_n}{|\Lambda|}\right).
\end{equation*}
For fixed $\Lambda$ and $\omega$, the map $\chi_{\omega,\Lambda}:\
\mathcal{M}_p(\R)\to\mathcal{M}_p(\R)$ is measurable as the map
$t\mapsto\chi_\Lambda(x,t)$ is. Moreover, by the computations made
above (see~\eqref{eq:78} and~\eqref{eq:77}), one has
\begin{equation}
  \label{eq:79}
  \chi_{\omega,\Lambda}(\Xi_J(\omega,\Lambda))=
  \tilde\Xi_J(\omega,\Lambda).
\end{equation}
For any $x\in R$, $t$ almost surely, one has $\chi_\Lambda(x;t)\to x$
as $|\Lambda|\to+\infty$. Hence, as $|\Lambda|\to+\infty$,
$\chi_{\omega,\Lambda}$ tends to the identity except on at most a set
of measure $0$ in $\mathcal{M}_p(\R)$. On the other hand,
Theorem~\ref{thr:8} tells us that, $\omega$ almost surely,
$\Xi_J(\omega,\Lambda)$ converges in law to the Poisson process of
intensity $1$ on the real line. Thus, we can apply~\cite[Theorem
5.5]{MR1700749} to obtain that, $\omega$-almost surely,
$\tilde\Xi_J(\omega,\Lambda)$, that is,
$\tilde\Xi_J(\omega,t,\Lambda)$ under the measure $\nu_J(t)dt$,
converges in law to the Poisson process of intensity $1$ on the real
line. This completes the proof of Theorem~\ref{thr:8}.\qed
\subsection{The proof of Theorem~\ref{thr:7}}
\label{sec:proof-theorem-4}
To complete this proof, recalling the notations of
Theorem~\ref{thr:7}, we notice that, for $x>0$,
\begin{multline*}
  \left\{E_j\in J;\ \frac{|N(J)|}{|J|} |\Lambda|
    (E_{j+1}(\omega,\Lambda)-E_j(\omega,\Lambda))\geq x\right\}
  \\=\left\{E_j\in J;\ \nu(t) |\Lambda|
    (E_{j+1}(\omega,\Lambda)-E_j(\omega,\Lambda))\geq
    \nu_J(t)\cdot|J|\,x \right\}.
\end{multline*}
Thus, integration with respect to $\nu_J(t)dt$ over $J$,
Theorem~\ref{thr:8} and the same computations as those made to obtain
Proposition 4.4 in~\cite{MR2352280} lead to, $\omega$-almost surely
\begin{equation*}
  \begin{split}
    DLS(x;J,\omega,\Lambda)&=\int_J\frac{\#\left\{E_j\in J;\ \nu(t)
        |\Lambda| (E_{j+1}(\omega,\Lambda)-E_j(\omega,\Lambda))\geq
        \nu_J(t)\cdot|J|\,x \right\}}{N(J,\omega,\Lambda)}\nu_J(t)dt
    \\&\vers_{|\Lambda|\to+\infty}\int_Je^{-\nu_J(t)\cdot|J|\,x}\nu_J(t)dt.
  \end{split}
\end{equation*}
This completes the proof of Theorem~\ref{thr:7}.\qed

\def\cprime{$'$} \def\cydot{\leavevmode\raise.4ex\hbox{.}} \def\cprime{$'$}


\begin{thebibliography}{10}

\bibitem{MR2207021}
Michael Aizenman, Alexander Elgart, Serguei Naboko, Jeffrey~H. Schenker, and
  Gunter Stolz.
\newblock Moment analysis for localization in random {S}chr{\"o}dinger operators.
\newblock {\em Invent. Math.}, 163(2):343--413, 2006.

\bibitem{MR2002h:82051}
Michael Aizenman, Jeffrey~H. Schenker, Roland~M. Friedrich, and Dirk
  Hundertmark.
\newblock Finite-volume fractional-moment criteria for {A}nderson localization.
\newblock {\em Comm. Math. Phys.}, 224(1):219--253, 2001.
\newblock Dedicated to Joel L. Lebowitz.

\bibitem{MR2360226}
Jean~V. Bellissard, Peter~D. Hislop, and G{\"u}nter Stolz.
\newblock Correlation estimates in the {A}nderson model.
\newblock {\em J. Stat. Phys.}, 129(4):649--662, 2007.

\bibitem{MR1700749}
Patrick Billingsley.
\newblock {\em Convergence of probability measures}.
\newblock Wiley Series in Probability and Statistics: Probability and
  Statistics. John Wiley \& Sons Inc., New York, second edition, 1999.
\newblock A Wiley-Interscience Publication.

\bibitem{MR2180453}
Jean Bourgain and Carlos~E. Kenig.
\newblock On localization in the continuous {A}nderson-{B}ernoulli model in
  higher dimension.
\newblock {\em Invent. Math.}, 161(2):389--426, 2005.

\bibitem{MR2505733}
Jean-Michel Combes, Fran{{\c c}}ois Germinet, and Abel Klein.
\newblock Generalized eigenvalue-counting estimates for the {A}nderson model.
\newblock {\em J. Stat. Phys.}, 135(2):201--216, 2009.

\bibitem{MR2663411}
Jean-Michel Combes, Fran{{\c c}}ois Germinet, and Abel Klein.
\newblock Poisson statistics for eigenvalues of continuum random
  {S}chr{\"o}dinger operators.
\newblock {\em Anal. PDE}, 3(1):49--80, 2010.

\bibitem{MR2352267}
Fran{{\c c}}ois Germinet, Peter~D. Hislop, and Abel Klein.
\newblock Localization at low energies for attractive {P}oisson random
  {S}chr{\"o}dinger operators.
\newblock In {\em Probability and mathematical physics}, volume~42 of {\em CRM
  Proc. Lecture Notes}, pages 153--165. Amer. Math. Soc., Providence, RI, 2007.

\bibitem{MR2314108}
Fran{{\c c}}ois Germinet, Peter~D. Hislop, and Abel Klein.
\newblock Localization for {S}chr{\"o}dinger operators with {P}oisson random
  potential.
\newblock {\em J. Eur. Math. Soc. (JEMS)}, 9(3):577--607, 2007.

\bibitem{MR2078370}
Fran{{\c c}}ois Germinet and Abel Klein.
\newblock A characterization of the {A}nderson metal-insulator transport
  transition.
\newblock {\em Duke Math. J.}, 124(2):309--350, 2004.

\bibitem{MR2203782}
Francois Germinet and Abel Klein.
\newblock New characterizations of the region of complete localization for
  random {S}chr{\"o}dinger operators.
\newblock {\em J. Stat. Phys.}, 122(1):73--94, 2006.

\bibitem{Ge-Kl:10}
Fran{{\c c}}ois Germinet and Fr{\' e}d{\'e}ric Klopp.
\newblock Spectral statistics for random {Schr{\"o}dinger} operators in the
  localized regime.
\newblock ArXiv \url{http://arxiv.org/abs/1011.1832}, 2010.

\bibitem{MR2423576}
F.~Ghribi, P.~D. Hislop, and F.~Klopp.
\newblock Localization for {S}chr{\"o}dinger operators with random vector
  potentials.
\newblock In {\em Adventures in mathematical physics}, volume 447 of {\em
  Contemp. Math.}, pages 123--138. Amer. Math. Soc., Providence, RI, 2007.

\bibitem{MR2658987}
Fatma Ghribi and Fr{\'e}d{\'e}ric Klopp.
\newblock Localization for the random displacement model at weak disorder.
\newblock {\em Ann. Henri Poincar{\'e}}, 11(1-2):127--149, 2010.

\bibitem{MR2290333}
Gian~Michele Graf and Alessio Vaghi.
\newblock A remark on the estimate of a determinant by {M}inami.
\newblock {\em Lett. Math. Phys.}, 79(1):17--22, 2007.

\bibitem{MR2509108}
Peter~D. Hislop.
\newblock Lectures on random {S}chr{\"o}dinger operators.
\newblock In {\em Fourth {S}ummer {S}chool in {A}nalysis and {M}athematical
  {P}hysics}, volume 476 of {\em Contemp. Math.}, pages 41--131. Amer. Math.
  Soc., Providence, RI, 2008.

\bibitem{MR1934351}
Peter~D. Hislop and Fr{\'e}d{\'e}ric Klopp.
\newblock The integrated density of states for some random operators with
  nonsign definite potentials.
\newblock {\em J. Funct. Anal.}, 195(1):12--47, 2002.

\bibitem{MR2509110}
Werner Kirsch.
\newblock An invitation to random {S}chr{\"o}dinger operators.
\newblock In {\em Random {S}chr{\"o}dinger operators}, volume~25 of {\em Panor.
  Synth{\`e}ses}, pages 1--119. Soc. Math. France, Paris, 2008.
\newblock With an appendix by Fr{\'e}d{\'e}ric Klopp.

\bibitem{MR2307751}
Werner Kirsch and Bernd Metzger.
\newblock The integrated density of states for random {S}chr{\"o}dinger
  operators.
\newblock In {\em Spectral theory and mathematical physics: a {F}estschrift in
  honor of {B}arry {S}imon's 60th birthday}, volume~76 of {\em Proc. Sympos.
  Pure Math.}, pages 649--696. Amer. Math. Soc., Providence, RI, 2007.

\bibitem{MR95m:82080}
Fr{\'e}d{\'e}ric Klopp.
\newblock Localization for some continuous random {S}chr{\"o}dinger operators.
\newblock {\em Comm. Math. Phys.}, 167(3):553--569, 1995.

\bibitem{1007.2483}
Fr{\'e}d{\'e}ric Klopp, Michael Loss, Shu Nakamura, and G{\"u}nter Stolz.
\newblock Localization for the random displacement model, 2010.
\newblock ArXiv \url{http://fr.arxiv.org/abs/arXiv:1007.2483}.

\bibitem{Mi:06}
Nariyuki Minami.
\newblock The energy level statistics for the {Anderson} tight binding model -
  statement of a conjecture.
\newblock Online document:
  \href{http://www.math.h.kyoto-u.ac.jp/~ueki/SR06/minami.pdf}{\tt
  www.math.h.kyoto-u.ac.jp/\~\,ueki/SR06/minami.pdf}.

\bibitem{MR97d:82046}
Nariyuki Minami.
\newblock Local fluctuation of the spectrum of a multidimensional {A}nderson
  tight binding model.
\newblock {\em Comm. Math. Phys.}, 177(3):709--725, 1996.

\bibitem{MR2352280}
Nariyuki Minami.
\newblock Theory of point processes and some basic notions in energy level
  statistics.
\newblock In {\em Probability and mathematical physics}, volume~42 of {\em CRM
  Proc. Lecture Notes}, pages 353--398. Amer. Math. Soc., Providence, RI, 2007.

\bibitem{Mi:11}
Nariyuki Minami.
\newblock Energy level statistics: a formulation and some examples.
\newblock In N.~Minami, editor, {\em Spectra of random operators and related
  topics}, 2011.
\newblock To appear.

\bibitem{MR94h:47068}
Leonid Pastur and Alexander Figotin.
\newblock {\em Spectra of random and almost-periodic operators}, volume 297 of
  {\em Grundlehren der Mathematischen Wissenschaften [Fundamental Principles of
  Mathematical Sciences]}.
\newblock Springer-Verlag, Berlin, 1992.

\bibitem{MR2364939}
Sidney~I. Resnick.
\newblock {\em Extreme values, regular variation and point processes}.
\newblock Springer Series in Operations Research and Financial Engineering.
  Springer, New York, 2008.
\newblock Reprint of the 1987 original.

\bibitem{MR1935594}
Peter Stollmann.
\newblock {\em Caught by disorder}, volume~20 of {\em Progress in Mathematical
  Physics}.
\newblock Birkh{\"a}user Boston Inc., Boston, MA, 2001.
\newblock Bound states in random media.

\bibitem{MR2378428}
Ivan Veseli{\'c}.
\newblock {\em Existence and regularity properties of the integrated density of
  states of random {S}chr{\"o}dinger operators}, volume 1917 of {\em Lecture
  Notes in Mathematics}.
\newblock Springer-Verlag, Berlin, 2008.

\end{thebibliography}
\end{document}